\documentclass{amsart} 

\usepackage{amsmath,amssymb,amsthm,amsfonts,amsxtra,mathtools}
\usepackage{enumerate,verbatim,mathrsfs,comment,color,url}

\usepackage[all,2cell,ps]{xy}
\usepackage[pagebackref]{hyperref}

\usepackage{graphicx}
\usepackage{tikz-cd}
\usepackage{cleveref}
\usepackage{verbatim}

\DeclareMathOperator{\ann}{ann}

\DeclareMathOperator{\codim}{codim}

\DeclareMathOperator{\curv}{curv}
\DeclareMathOperator{\cx}{cx}
\DeclareMathOperator{\depth}{depth}
\DeclareMathOperator{\Der}{Der}

\DeclareMathOperator{\embdim}{embdim}

\DeclareMathOperator{\Ext}{Ext}

\DeclareMathOperator{\Hom}{Hom}
\DeclareMathOperator{\htt}{ht}
\DeclareMathOperator{\id}{id}
\DeclareMathOperator{\Image}{Image}

\DeclareMathOperator{\injcx}{inj\,cx}
\DeclareMathOperator{\injcurv}{inj\,curv}

\DeclareMathOperator{\Min}{Min}

\DeclareMathOperator{\pd}{pd}

\DeclareMathOperator{\rank}{rank}

\DeclareMathOperator{\Spec}{Spec}
\DeclareMathOperator{\Soc}{Soc}

\DeclareMathOperator{\tcurv}{tcurv}
\DeclareMathOperator{\tcx}{tcx}
\DeclareMathOperator{\Tor}{Tor}
\DeclareMathOperator{\Tot}{Tot}

\DeclareMathOperator{\type}{type}

\renewcommand{\ge}{\geqslant}
\renewcommand{\le}{\leqslant}

\newcommand{\fm}{\mathfrak{m}}
\newcommand{\fn}{\mathfrak{n}}
\newcommand{\fp}{\mathfrak{p}}
\newcommand{\fq}{\mathfrak{q}}
\renewcommand{\iff}{if and only if }

\theoremstyle{plain}
\newtheorem{theorem}{Theorem}[section]
\newtheorem{lemma}[theorem]{Lemma}
\newtheorem{proposition}[theorem]{Proposition}
\newtheorem{corollary}[theorem]{Corollary}

\newtheorem{innercustomtheorem}{Theorem}
\newenvironment{customtheorem}[1]
{\renewcommand\theinnercustomtheorem{#1}\innercustomtheorem}
{\endinnercustomtheorem}

\newtheorem{innercustomcorollary}{Corollary}
\newenvironment{customcorollary}[1]
{\renewcommand\theinnercustomcorollary{#1}\innercustomcorollary}
{\endinnercustomcorollary}

\theoremstyle{definition}
\newtheorem{definition}[theorem]{Definition}
\newtheorem{conjecture}[theorem]{Conjecture}

\newtheorem{example}[theorem]{Example}

\newtheorem{setup}[theorem]{Setup}

\theoremstyle{remark}
\newtheorem{remark}[theorem]{Remark}

\numberwithin{equation}{section}

\title[Complexity and curvature of (pairs of) Cohen-Macaulay modules]{Complexity and curvature of (pairs of) Cohen-Macaulay modules, and their applications}

\author[S.~Dey]{Souvik Dey}
\address{Department of Mathematical Sciences, 850 West Dickson Street, University of Arkansas, Fayetteville, Arkansas 72701}
\email{souvikd@uark.edu} 
\urladdr{\url{https://orcid.org/0000-0001-8265-3301}}

\author[D.~Ghosh]{Dipankar Ghosh}
\address{Department of Mathematics, Indian Institute of Technology Kharagpur, West Bengal - 721302, India}
\email{dipankar@maths.iitkgp.ac.in, dipug23@gmail.com}
\urladdr{\url{https://orcid.org/0000-0002-3773-4003}}

\author[A.~Saha]{Aniruddha Saha}%
\address{Department of Mathematics, Indian Institute of Technology Hyderabad, Kandi, Sangareddy - 502285, Telangana, India}
\email{ma20resch11001@iith.ac.in, sahaa43@gmail.com}


\begin{document}

\pagenumbering{arabic}
\thispagestyle{empty}

\begin{abstract}
    The complexity and curvature of a module, introduced by Avramov, measure the growth of Betti and Bass numbers of a module, and distinguish the modules of infinite homological dimension. The notion of complexity was extended by Avramov-Buchweitz to pairs of modules that measure the growth of Ext modules. The related notion of Tor complexity was first studied by Dao. Inspired by these notions, we define Ext and Tor curvature of pairs of modules. The aim of this article is to study (Ext and Tor) complexity and curvature of pairs of certain CM (Cohen-Macaulay) modules, and establish lower bounds of complexity and curvature of pairs of modules in terms of that of a single module. It is known that among all modules, the residue field has maximal complexity and curvature, moreover they characterize complete intersection local rings. As applications of our results, we provide some upper bounds of the curvature of the residue field in terms of curvature and multiplicity of any nonzero CM module. As a final upshot, these allow us to characterize complete intersection local rings (including hypersurfaces and regular rings) in terms of complexity and curvature of pairs of certain CM modules. In particular, under some additional hypotheses, we characterize complete intersection and regular local rings via injective curvature of the ring and that of the module of K\"{a}hler differentials respectively. Thus, we make partial progress towards a question of Christensen-Striuli-Veliche, as well as another by Vasconcelos.
\end{abstract}
\maketitle

\section{Introduction}

The notions of (projective and injective) complexity and curvature of a module were introduced by Avramov in \cite{Avr89a}, \cite{Avr89b} and \cite{Avr96}. These invariants measure the growth of Betti and Bass numbers of a module, and distinguish the modules of infinite projective and injective dimension. Later, in \cite{AB00}, the notion of complexity was extended by Avramov-Buchweitz to each pair of modules (\Cref{defn:cx-M-N}), which measures the polynomial growth rate of Ext of the pair of modules. The related notion of Tor complexity of a pair of modules (\Cref{defn:tcx-tcurv}), was first studied by Dao in \cite{Dao}. Inspired by these notions, we define Ext and Tor curvature of a pair of modules (Definitions~\ref{defn:curv-M-N} and \ref{defn:tcx-tcurv}). Complexity and curvature of a module can be described as that of a pair of modules by taking one of them as the residue field, see \Cref{defn:cx-curv}. The aim of this article is to study (Ext and Tor) complexity and curvature of pairs of certain CM (Cohen-Macaulay) modules, and establish lower bounds of complexity and curvature of pairs of modules in terms of that of a single module. It is well known that the residue field has maximal complexity and curvature, which characterizes complete intersection local rings. As applications of our results, we provide some upper bounds of the curvature of the residue field in terms of curvature and multiplicity of any nonzero CM module. Moreover, for a large class of CM modules, we show that their (injective) complexity and that of the residue field differ at most by $1$. Thus, we obtain characterizations of complete intersection local rings (including hypersurfaces and regular rings) in terms of complexity and curvature of pairs of certain CM modules. Along the way, we also establish a Tor version of a result \cite[Thm.~II.(3)]{AB00} of Avramov-Buchweitz, see \Cref{prop:AB-like-result-on-tcx}.(1). As some other applications, we provide partial results towards understanding rings whose injective curvature (as a module over itself) is at most $1$ or for which the injective curvature  of the module of K\"{a}hler differentials is at most $1$. In particular, \Cref{thm-C}.(2) provides partial progress towards \cite[Ques.~1.3]{CSV} and a question of Huneke (see \cite[Ques., p.~647]{JL07}), while \Cref{thm-D} shows that the conjectures of Vasconcelos \cite[Conj.~($\operatorname{C}_2$), p.~1807]{Va78} and Berger \cite{berg} are affirmative in certain cases.

\begin{setup}
	Throughout, unless otherwise specified, let $(R,\fm,k)$ be a commutative Noetherian local ring, and all $R$-modules are assumed to be finitely generated.
\end{setup}

As is often standard, we usually refer to projective complexity and curvature simply as complexity and curvature moving forward.

Our first main result compares the complexity (resp., curvature) of a pair of modules and (injective) complexity (resp., curvature) of a single module under some annihilation condition of Tor and Ext.

\begin{customtheorem}{A}[See \Cref{thm:main-cx-curv-CM-mod} and \Cref{cor:mcm-aprrox}]\label{thm-A}
	Let $M$ and $X$ be nonzero $R$-modules such that $M$ is CM. Then, the following hold.
\begin{enumerate}[\rm(1)]
    \item 
    Suppose there exists a positive integer $h$ satisfying $\fm^h \Tor_n^R(X,M)=0$ for all $n \gg 0$. Then
    \begin{enumerate}[\rm (i)]
        \item $\curv_R(X) \le \sup\left\{\frac{e(M)}{\mu(M)}-1, \tcurv_R(X,M)\right\}.$
        \item If $e(M)\le 2 \mu(M) $, then $\cx_R(X) \le 1+\tcx_R(X,M)$. 
        \item 
        If $e(M)< 2\mu(M) $, then $\cx_R(X) \le \tcx_R(X,M)$ and $\curv_R(X) \le \tcurv_R(X,M)$.
     \end{enumerate}
    \item 
    Suppose there exists a positive integer $h$ satisfying $\fm^h \Ext^n_R(X,M)=0$ for all $n \gg 0$. Then
    \begin{enumerate}[\rm (i)]
        \item $\curv_R(X)\le \sup\left\{\frac{e(M)}{\type(M)}-1, \curv_R(X,M)\right\}.$
        \item 
        If $e(M)\le 2\type(M) $, then $\cx_R(X) \le 1+\cx_R(X,M)$.
        \item 
        If $e(M)< 2\type(M) $, then $\cx_R(X) \le \cx_R(X,M)$ and $\curv_R(X) \le \curv_R(X,M)$.
    \end{enumerate}
    \item
    Suppose that $R$ or $X$ is also CM, and there exists an integer $h>0$ satisfying $\fm^h \Ext^n_R(M,X)=0$ for all $n \gg 0$. Then
    \begin{enumerate}[\rm (i)]
        \item $\injcurv_R(X) \le \sup\left\{ \frac{e(M)}{\mu(M)}-1, \curv_R(M,X)\right\}$.
        \item If $e(M)\le 2\mu(M) $, then $\injcx_R(X)\le 1+\cx_R(M,X)$.
        \item If $e(M)<2\mu(M) $, then $\injcx_R(X)\le \cx_R(M,X)$ and  $\injcurv_R(X)\le \curv_R(M,X)$.
    \end{enumerate}
\end{enumerate}
\end{customtheorem}

Among all $R$-modules, it is well known that the residue field (and hence a nonzero $R$-module that is annihilated by $\fm$) has the extremal (injective) complexity as well as curvature (cf.~Proposition~\ref{prop:cx-curv-facts}.\eqref{cx-curv-M-k}). Taking $X=k$ in Theorem~\ref{thm-A}, one obtains the following consequence.

\begin{customcorollary}{\ref{cor:cx-curv-k-CM-mod}}\label{cor:1st-intro}
Let $M$ be a nonzero CM $R$-module. Then, the following hold.
\begin{enumerate}[\rm(1)]
    \item $\curv_R(k) \le \sup\left\{\frac{e(M)}{\mu(M)}-1, \curv_R(M)\right\}.$
    \item $\curv_R(k) \le \sup\left\{\frac{e(M)}{\type(M)}-1, \injcurv_R(M)\right\}.$
    \item If $e(M)\le 2\mu(M) $, then $\cx_R(k)\le 1+\cx_R(M)$, i.e., $\cx_R(k)=\cx_R(M)$ or $1+\cx_R(M)$.
    \item If $e(M)<2\mu(M) $, then $\cx_R(k) = \cx_R(M)$ and  $\curv_R(k) = \curv_R(M)$.
    \item If $e(M)\le 2\type(M) $, then $\cx_R(k)\le 1+\injcx_R(M)$, i.e., $\cx_R(k)=\injcx_R(M)$ or $1+\injcx_R(M)$.
    \item If $e(M)<2\type(M) $, then $\cx_R(k)=\injcx_R(M)$ and  $\curv_R(k)=\injcurv_R(M)$.
\end{enumerate}
\end{customcorollary}




For modules of minimal multiplicity (cf.~\Cref{defn:min-mult}), we obtain similar results (namely, \Cref{thm:cx-curv-min-mult} and \Cref{Cor:cx-curv-min-CM}) that compliment \Cref{thm-A}. In view of Lemma~\ref{lem:mult-inequality}.(2) and (3),  when the residue field $k$ is infinite, a CM $R$-module $M$ has minimal multiplicity if and only if $\fm^2M = ({\bf x}) \fm M$ for some system of parameters ${\bf x}$ of $M$. In particular, every Ulrich $R$-module (cf.~\Cref{defn:Ulrich-mod}) has minimal multiplicity, and every MCM module over a CM ring of minimal multiplicity has minimal multiplicity; see \Cref{rmk:Ulrich-min-mult} and \Cref{rmk:min-mult-R-M}. Taking $X=k$ in \Cref{thm:cx-curv-min-mult}, one obtains \Cref{cor:main-cx-curv-CM-min-mult}, which is the counterpart of \Cref{cor:1st-intro} for modules of minimal multiplicity.

In \Cref{thm:e-mu-type}, we provide a new characterization of Ulrich modules in terms of type. This particularly ensures that the denominator $e(M)-\type(M)$ in \Cref{thm:cx-curv-min-mult}.(3).(i) and (4).(i), \Cref{Cor:cx-curv-min-CM}.(i) and \Cref{cor:main-cx-curv-CM-min-mult}.(2) is nonzero exactly when $M$ is not Ulrich.

In regards to \Cref{thm-A}, \Cref{thm:cx-curv-min-mult} and \Cref{Cor:cx-curv-min-CM}, we also draw the reader's attention to the fact that the annihilation condition on $\Tor$ and $\Ext$ are readily satisfied if $X$ or $M$ locally has finite projective dimension on the punctured spectrum, so in particular if $R$ has isolated singularity. We refer to \Cref{lem:Tor-m-zero} for more details, and Corollaries~\ref{cor:loc-fin-pd} and \ref{cor:loc-fin-pd-min} respectively for the precise statements on this. 


An interesting upshot of Corollaries~\ref{cor:1st-intro} and \ref{cor:main-cx-curv-CM-min-mult} is that every nonzero CM module $M$ with $ e(M) \le 2 \mu(M) $ or $e(M)\le 2\type(M) $ and every nonzero module $M$ of minimal multiplicity satisfying $e(M)\ge 2\mu(M)$ or $e(M)\ge 2\type(M)$ can be used to characterize complete intersection local rings; see Corollaries~\ref{cor:char-CI-rings} and \ref{cor:char-CI-minmult-rings} in this regard.

We construct Examples~\ref{exam-2-cx-curv-inequality}, \ref{exam-1-cx-curv-inequality}, \ref{exam:strict-ineq-cant-be-eq} and \ref{exam-3-cx-curv-inequality} regarding sharpness of the inequalities and necessity of the hypotheses in Corollaries~\ref{cor:cx-curv-k-CM-mod} and \ref{cor:main-cx-curv-CM-min-mult} (hence in \Cref{thm-A}, \Cref{thm:cx-curv-min-mult} and \Cref{Cor:cx-curv-min-CM}). See Remarks~\ref{rmk:3-examples} and \ref{rmk:exam-sec-6} for more details. 

Our next result shows, in particular, that modules of minimal multiplicity have trivial Ext vanishing. Moreover, for Gorenstein rings, the existence of such a nonzero module with eventual self-Ext vanishing forces the ring to be a hypersurface. We note that there are many results in the literature on trivial Ext vanishing, see for examples, \cite[Thm.~4.2]{AB00}, \cite[4.1, 5.3 and 5.4]{AINS}, \cite[Thm.~4.1.(2)]{HSV04}, \cite[Cor.~3.3]{JS04}, \cite[Thms.~3.2 and 4.2]{LM20} and \cite[Thm.~6.3]{DG23}.

\begin{customtheorem}{\ref{thm:trivial-Ext-van}}
    Let $M$ be a nonzero $R$-module of minimal multiplicity.
\begin{enumerate}[\rm (1)]
    \item If there exists an integer $h\ge 0$ such that $\fm^h\Ext_R^{\gg 0}(M,M)=0$ $($e.g., $M$ is locally of finite projective dimension on the punctured spectrum, {\rm cf.~\Cref{lem:Tor-m-zero}}$)$, then
    $$\inf\{\cx_R(M),\injcx_R(M)\}\le \cx_R(M,M) \text{ and } \inf\{\curv_R(M), \injcurv_R(M)\}\le  \curv_R(M,M)$$ 
    \item In particular, if $\Ext_R^{\gg 0}(M,M)=0$, then either $\pd_R(M)<\infty$ or $\id_R(M)<\infty$.
    \item If $R$ is Gorenstein and $\Ext_R^{\gg 0}(M,M)=0$, then $R$ is hypersurface.
\end{enumerate}
\end{customtheorem}

As applications of Theorem~\ref{thm-A} and \Cref{thm:cx-curv-min-mult}, in \Cref{sec:char}, we provide various criteria for a local ring to be complete intersection in terms of Ext and Tor complexity and curvature of certain pairs of CM modules. We just mention here two such criteria.

\begin{customtheorem}{B}[See Theorems~\ref{thm:cx-M-N-inequa} and \ref{thm:tcx-M-N-inequa} for more related results]\label{thm-B}{~}
\begin{enumerate}[\rm (1)]
    \item 
    Let $M$ and $N$ be nonzero CM $R$-modules such that $ e(M) \le 2 \mu(M) $, $ e(N) \le 2 \type(N) $, and $\fm^h\Ext_R^{\gg 0}(M,N)=0$ for some $h\ge 0$. Then
    \[
        \sup\big\{\cx_R(M),\injcx_R(M)\big\}+\sup\big\{\injcx_R(N),\cx_R(N)\big\} \le 2\big(1+\cx_R(M,N)\big).
    \]
    Moreover, the following statements are equivalent:
    \begin{enumerate}[\rm (a)]
        \item $R$ is complete intersection $($of codimension at most $2+\cx_R(M,N)$$)$.
        \item $\cx_R(M,N)<\infty$.
        \item $\curv_R(M,N)\le 1$.
    \end{enumerate}
    \item 
    Let $M$ and $N$ be nonzero CM $R$-modules such that $ e(M) \le 2 \mu(M) $, $ e(N) \le 2 \mu(N) $, and $\fm^h\Tor_{\gg 0}^R(M,N) = 0$ for some $h\ge 0$. Then
    \[
        \cx_R(M) + \cx_R(N) \le 2\big(1+\tcx_R(M,N)\big).
    \]
    Moreover, the following statements are equivalent:
    \begin{enumerate}[\rm (a)]
        \item $R$ is complete intersection $($of codimension at most $2+\tcx_R(M,N)$$)$.
        \item $\tcx_R(M,N)<\infty$.
        \item $\tcurv_R(M,N)\le 1$.
    \end{enumerate}
\end{enumerate}
\end{customtheorem}

When exactly one of the two inequalities (resp., both inequalities) in each of the hypotheses of \Cref{thm-B}.(1) and (2) is a strict inequality (resp., are strict inequalities), the conclusions of \Cref{thm-B}.(1) and (2) can be improved further. We refer to Theorems~\ref{thm:cx-M-N-inequa-2} and \ref{thm:tcx-M-N-inequa-2} (resp., \ref{thm:cx-M-N-inequa-3} and \ref{thm:tcx-M-N-inequa-3}) for the details. Note that $\Ext_R^{\gg 0}(M,N)=0$ \iff $\cx_R(M,N)=0$, and $\Tor_{\gg 0}^R(M,N)=0$ \iff $\tcx_R(M,N)=0$. Thus these theorems
particularly characterize complete intersection local rings of codimension at most $2$, hypersurfaces and regular local rings, see Corollaries~\ref{cor:cx-M-N-inequa}, \ref{cor:cx-M-N-inequa-2}, \ref{cor:cx-M-N-inequa-3}, \ref{cor:tcx-M-N-inequa}, \ref{cor:tcx-M-N-inequa-2} and \ref{cor:tcx-M-N-inequa-3}.

It is standard that $R$ is Gorenstein if and only if the Bass numbers of $R$ are eventually $0$ (equivalently, $\injcurv_R(R)=0$). More generally, it was shown by Foxby that if there exists a nonzero module of finite projective dimension whose Bass numbers are eventually $0$, then the base ring is Gorenstein (cf.~\cite[Cor.~4.4]{Fo77}). This raises a natural question of whether `mild growth rate' of the Bass numbers of the ring forces the ring to be Gorenstein or not. Interpreting `mild' as close to polynomial growth, this question was studied in \cite{JL07, CSV}. Our next result takes another step in this direction. 

\begin{customtheorem}{C}[See \Cref{prop:pd-injcurv-max} and \Cref{cor:exp-growth-JL}.(2), (3)]\label{thm-C}{~}
\begin{enumerate}[\rm (1)]
    \item 
    Let $M$ be a nonzero $R$-module. 
    \begin{enumerate}[\rm (a)]
        \item 
        If $\pd_R(M) < \infty$, then $\injcx_R(M)=\injcx_R(R)$ and $\injcurv_R(M)=\injcurv_R(R)$.
        \item 
        If $\id_R(M) < \infty$, then $\cx_R(M)=\injcx_R(R)$ and $\curv_R(M)=\injcurv_R(R)$.
    \end{enumerate}
    \item 
    If $R$ is CM such that $e(R)\le 2\type(R)$, then $R$ is complete intersection when at least one of the following holds.
    \begin{enumerate}[\rm (a)]
        \item There exists a nonzero $R$-module $M$ of finite projective dimension such that $\injcurv_R(M)\le 1$.
        \item
        There exists a nonzero $R$-module $M$ of finite injective dimension such that $\curv_R(M)\le 1$.
    \end{enumerate}
\end{enumerate}
\end{customtheorem}

The classical Jacobian criterion characterizes the regularity of an affine algebra $R$ over a field $k$ of characteristic zero in terms of the projective dimension of its module of K\"{a}hler differentials, denoted by $\Omega_{R/k}$. In this theme, Avramov-Herzog \cite[Thm.~2]{AH} showed that such an algebra is complete intersection if $\curv_R(\Omega_{R/k})\le 1$. Motivated by these results, as applications of Theorem~\ref{thm-A} and \Cref{thm:cx-curv-min-mult}, we prove the following. This in particular gives a partial progress to a conjecture of Berger (see \Cref{Conj-B63}).

\begin{customtheorem}{D}[See \Cref{thm:diff} and Corollary~\ref{cor:Berger}]\label{thm-D}{~}
\begin{enumerate}[\rm (1)]
    \item
    Suppose $R$ is an equidimensional reduced local ring, the residue field $k$ of $R$ has characteristic $0$, and $R$ is either essentially of finite type over $k$ or it is a homomorphic image of a power series ring
    over $k$. Let $\Omega_{R/k}$ be CM. Then the following hold.
    \begin{enumerate}[\rm (a)]
        \item If $e(R) \dim(R) \le 2 \embdim(R)$ and $\curv_R(\Omega_{R/k})\le 1$, then $R$ is regular.
        \item 
        If $R$ is not a field, $\Omega_{R/k}$ has minimal multiplicity, and $e(R)\dim(R)\ge 2\embdim(R)$, then $\injcurv_R(\Omega_{R/k})> 1$.
    \end{enumerate}
    \item
    {\rm \Cref{Conj-B63}} due to Berger holds true when $e(R) \le 2 \embdim(R)$ and $\curv_R(\Omega_{R/k})\le 1$.
\end{enumerate}
\end{customtheorem}

We now describe briefly the contents of the article. In \Cref{sec:cx-curv-seq}, we recall the notion of complexity and curvature of sequences of non-negative real numbers, and prove various preliminary properties of them. In \Cref{sec:cx-curv-mod}, we recall Ext and Tor complexity, and define Ext and Tor curvature of a pair of modules based on the notions from \Cref{sec:cx-curv-seq}, and we also establish a connection between Ext and Tor complexity (resp., curvature), see \Cref{cx-duality}. This section contains many preparatory results about these invariants. In \Cref{sec:mult}, we recall the notion of (CM) modules of minimal multiplicity, and relate them to Ulrich modules (\Cref{prop:cm-min-ul}). Along the way, we also establish a new characterization of Ulrich modules (\Cref{thm:e-mu-type}). \Cref{sec:cx-curv-cm-mod} is devoted to proving the first main results of this paper, namely \Cref{thm-A} and \Cref{cor:1st-intro}. \Cref{sec:cx-curv-cm-mod-min-mult} proves \Cref{thm:cx-curv-min-mult} and Corollaries~\ref{Cor:cx-curv-min-CM} and \ref{cor:main-cx-curv-CM-min-mult}, which compliment the results of \Cref{sec:cx-curv-cm-mod}. \Cref{sec:char} contains various criteria for a local ring to be complete intersection, hypersurface or regular in terms of Ext and Tor complexity and curvature of certain pairs of CM modules. Section~\ref{sec:diff} deals with module of differentials, where we mainly prove \Cref{thm-D}.

Finally, it is worth mentioning that in two related follow-up papers \cite{DGS1, DGS2}, many of the results of Sections~\ref{sec:cx-curv-cm-mod} and \ref{sec:cx-curv-cm-mod-min-mult} have found further applications. 

\section{Complexity and curvature of sequences}\label{sec:cx-curv-seq}

In this section, we recall the notions of complexity and curvature of sequences of non-negative real numbers, and prove some results on these. These are heavily used in the rest of the paper.

\begin{definition}\label{defn:cx-curv-x_n}
	Let $\{a_n\}$ be a sequence of non-negative real numbers.
    \begin{enumerate}[(1)]
        \item 
        The complexity of $\{a_n\}$, denoted $\cx\big( \{a_n\} \big)$, is defined to be the smallest integer $b\ge 0$ such that $a_n \le \alpha n^{b-1}$ for some real number $\alpha>0$ and for all $n \gg 0$. If no such $b$ exists, then $\cx\big( \{a_n\} \big) :=\infty$.
        \item 
        The curvature of $\{a_n\}$ is defined by
	    $$\curv\big( \{a_n\} \big) := \limsup_{n\to\infty} \sqrt[n]{a_n}.$$
    \end{enumerate}
\end{definition}

\begin{remark}\label{rmk:cx-curv-ineq}
    The following facts are obvious from the definition of complexity and curvature.
    \begin{enumerate}[\rm (1)]
        \item If $a_n \le b_n$ for all $n\gg 0$, then $\cx(\{a_n \}) \le \cx(\{b_n \})$ and $\curv(\{a_n\}) \le \curv(\{b_n \})$.
        \item If $v>0$ be a real number, then $\cx(\{va_n \}) = \cx(\{a_n \})$ and $\curv(\{va_n \}) = \curv(\{a_n \})$.
        \item If $\cx(\{a_n\})<\infty$, then $\curv(\{a_n\})\le 1$.
    \end{enumerate}
\end{remark}

\begin{remark}\label{rmk:curv-zero}
    Let $\{a_n\}$ be a sequence of non-negative integers. Then, the following are equivalent:\\
    (1) $a_n=0$ for all $n\gg 0$; (2) $\cx(\{a_n\})=0$; and (3) $\curv\big( \{a_n\} \big)<1$.
\end{remark}

\begin{proof}
    The implications (1) $\Rightarrow$ (2) and (1) $\Rightarrow$ (3) are trivial.

    (2) $\Rightarrow$  (1): Let $\cx(\{a_n\})=0$. Then there exists $\alpha>0$ such that $0\le a_n\le \frac{\alpha}{n}$ for all $n\gg 0$. Hence, $0\le a_n<1$ for $n\gg 0$. Since $a_n$ are all integers, it follows that $a_n=0$ for all $n\gg 0$. 
    
    (3) $\Rightarrow$ (1): Let $\curv\big( \{a_n\} \big) = a < 1$. Choose $\epsilon>0$ such that $a+\epsilon <1$. Then, by the definition of limit supremum, for this $\epsilon$, there exists an integer $n_0$ such that $\sqrt[n]{a_n} < a+\epsilon$ for all $n \ge n_0$. Thus $a_n<(a+\epsilon)^n<1$ for all $n \ge n_0$, since $a+\epsilon <1$. Since $a_n$ are all non-engative integers, it follows that $a_n=0$ for all $n\gg 0$.
\end{proof}

The following lemmas on calculus of sequences are useful to get the desired (in)equalities on complexities and curvatures.

\begin{lemma}\label{lem:cx-curv-seq}
    Let $\{ a_n\}_{n=1}^{\infty}$ and $\{ b_n\}_{n=1}^{\infty}$ be two sequences of non-negative integers such that $a_n \le w(b_n+b_{n-1})$ for all $n \gg 1$, where $w$ is some positive real number. Then
    \[ 
        \cx(\{a_n \}) \le \cx(\{b_n \}) \quad \mbox{and} \quad \curv(\{a_n\}) \le \curv(\{b_n \}).
    \]
\end{lemma}

\begin{proof}
    If $\cx(\{b_n \})=\infty$, then there is nothing to prove. So we may assume that $\cx(\{b_n \})=b < \infty$. By definition, there exists $\alpha >0$ such that $b_n \le \alpha n^{b-1}$ for all $n\gg 1$. If $b=0$, then $b_n = 0$ for all $n\gg 1$, and hence $a_n = 0$ for all $n\gg 1$, i.e., $\cx(\{a_n\}) = 0$. So we may assume that $b\ge 1$. In this case, $a_n \le w(b_n + b_{n-1}) \le w\alpha (n^{b-1}+(n-1)^{b-1}) \le 2w\alpha n^{b-1}$ for all $n\gg 1$. Therefore $\cx(\{a_n \})\le b=\cx(\{b_n \})$.
    
    Assume that $\curv(\{a_n\})=a$ and $\curv(\{b_n\})=c$.
    Then, by definition, for every $\epsilon > 0$, one has that $b_n \le (c+ \epsilon)^n$ for all $n \gg 1$. Hence $a_n\le w(b_n + b_{n-1}) \le w(c+ \epsilon)^n+w(c+ \epsilon)^{(n-1)}$ for all $n \gg 1$. We consider two subcases. First, assume that $c+ \epsilon\ge 1$. Then $a_n\le 2w(c+ \epsilon)^n$, i.e., $\sqrt[n]{a_n} \le (c+\epsilon)\sqrt[n]{2w}$ for all $n\gg  1$, which implies that $a\le c+\epsilon$. In the other case, when $c+ \epsilon < 1$, then $a_n\le 2w(c+ \epsilon)^{n-1}$, i.e., $\sqrt[n]{a_n} \le (c+\epsilon)\sqrt[n]{\frac{2w}{c+ \epsilon}}$ for all $n\gg  1$, which yields that $a\le c+\epsilon$. Thus, in any case, $a\le c+\epsilon$, and this holds true for every $\epsilon > 0$. So $a\le c$, i.e., $\curv(\{a_n\}) \le \curv(\{b_n \})$.
\end{proof}

\begin{lemma}\label{lem:curv-x-y}
Let $\{x_n\}$ and $\{y_n\}$ be sequences of non-negative real numbers. Let $a>0$ and $b$ be non-negative real numbers such that $x_{n+1}\le bx_n+ay_n$ for all $n\gg 1$. Then, $\curv(\{x_n\})\le \sup\big\{b, \curv(\{y_n\})\big\}$. 
\end{lemma}

\begin{proof}
Since $\curv(\{y_n\})=\curv(\{ay_n\})$, we may assume that $a=1$. Thus $x_{n+1}\le bx_n+y_n$ for all $n\gg 1$. Set $y:=\curv(\{y_n\})$. If $y=\infty$, then there is nothing to prove. Also, if $b=0$, then $x_{n+1}\le y_n$ for all $n\gg 1$, and hence $\curv(\{x_n\})\le \curv(\{y_n\})$. So we may assume that $y<\infty$ and $b>0$. Depending on $b$ and $y$, two cases may appear. 

{\bf Case 1}. Let $b>y$. Choose $0<\epsilon<b-y$ and fix it. Then, there exists $n_0>0$ such that $y_i\le (y+\epsilon)^i$ for all $i\ge n_0$ and  $x_{n+1}\le bx_n+y_n$ for all $n\ge n_0$. Consequently, for all $n\ge n_0$, one has that
\begin{align}
    x_{n+1} \le bx_n+y_n \le b^2 x_{n-1}+by_{n-1}+y_n &\le b^3 x_{n-2} + b^2y_{n-2}+by_{n-1}+y_n \label{eqn:xnplus1}\\
    & \le \cdots \le b^{n-n_0+1}x_{n_0}+\sum_{i=n_0}^n b^{n-i}y_i.\nonumber
\end{align}
Moreover, $\sum_{i=n_0}^n b^{n-i}y_i\le \sum_{i=n_0}^n b^{n-i}(y+\epsilon)^i=b^n\sum_{i=n_0}^n \left(\frac{y+\epsilon}{b}\right)^i\le b^n(n-n_0+1)$, since $y+\epsilon < b$. Hence, \eqref{eqn:xnplus1} yields that $x_{n+1}\le b^{n-n_0+1}x_{n_0}+b^n(n-n_0+1)=b^n(b^{1-n_0}x_{n_0}+(n-n_0+1))$ for all $n\ge n_0$. Therefore, since $\lim_{n\to \infty} (n+c)^{1/n}=1$ for any constant $c$, one obtains that $\lim\sup x_n^{1/n}\le b=\sup\{b,y\}$.

{\bf Case 2}. Assume that $b\le y$. Fix $\epsilon >0$. Then, $b<y+\epsilon$. There exists $n_0>0$ such that $y_i\le (y+\epsilon)^i$ for all $i\ge n_0$ and  $x_{n+1}\le bx_n+y_n$ for all $n\ge n_0$. As in Case 1, $x_{n+1}\le b^{n-n_0+1}x_{n_0}+\sum_{i=n_0}^n b^{n-i}y_i$ for all $n\ge n_0$. Furthermore, $\sum_{i=n_0}^n b^{n-i}y_i\le b^n \sum_{i=n_0}^n\left(\frac{y+\epsilon}{b}\right)^i \le b^n(n-n_0+1) \left(\frac{y+\epsilon}{b}\right)^n=(n-n_0+1)(y+\epsilon)^n$, where we have used that $y+\epsilon>b$, and hence the sequence $\{\left(\frac{y+\epsilon}{b}\right)^i\}_{i\ge 1}$ is increasing. Thus, $ x_{n+1} \le b^{n-n_0+1}x_{n_0}+(n-n_0+1)(y+\epsilon)^n \le (y+\epsilon)^n(b^{1-n_0}x_{n_0}+n-n_0+1)$ for all $n\ge n_0$, where the last inequality follows as $b< (y+\epsilon)$. Finally, since $\lim_{n\to \infty} (n+c)^{1/n}=1$ for any constant $c$, it follows that $\lim\sup x_n^{1/n}\le y+\epsilon$. Since $\epsilon >0$ was arbitrary, this proves that $\lim\sup x_n^{1/n}\le y=\sup\{b,y\}$.   
\end{proof}

\begin{lemma}\label{lem:curv-sum}
Let $\{x_n\}$ be a sequence of non-negative integers, and $a, b$ be positive real numbers. Then
$$\curv(\{ax_{n+1}+bx_n\})=\curv(\{x_n\}).$$
\end{lemma}

\begin{proof} 
    Let $y_n=ax_{n+1}+bx_n$. Then $y_n\ge bx_n$ for all $n\ge 0$. Hence, by Remark~\ref{rmk:cx-curv-ineq}, $\curv(\{y_n\})\ge \curv(\{x_n\})$. Set $x:=\curv(\{x_n\})$. Fix $\epsilon>0$. Then there exists $n_0$ such that $x_n \le (x+\epsilon)^n$ for all $n\ge n_0$. So, by setting $\alpha := a(x+\epsilon)+b$, one obtains that $y_n=ax_{n+1}+bx_n\le \alpha (x+\epsilon)^n$ for all $n\ge n_0$. It follows that $\lim\sup y_n^{1/n}\le x+\epsilon$. This holds for every $\epsilon>0$. Therefore $\curv(\{y_n\}) \le x = \curv(\{x_n\})$. This completes the proof.
\end{proof}

\begin{lemma}\label{lem:sum-curv-ieq}
    Let $\{x_n\}$ and $\{y_n\}$ be non-negative integers. Then
    \begin{enumerate}[\rm (1)]
        \item $\cx(\{x_{n}+ y_n\}) \le \sup\big\{ \cx(\{x_n\}), \cx(\{y_n\})\big\}$.
        \item 
        $\curv(\{x_{n}+ y_n\}) \le \sup\big\{ \curv(\{x_n\}), \curv(\{y_n\})\big\}$.
    \end{enumerate}
\end{lemma}
\begin{proof}
   If $\cx(\{x_n\})= \infty$ or $\cx(\{y_n\})= \infty$, then there is nothing to prove. So we may assume that $\cx(\{x_n\})= x$, $\cx(\{y_n\})= y$, and both are finite. Without loss of generality, one may assume that $x\le y$. Since $\cx(\{x_n\})= x$ and $\cx(\{y_n\})= y$, by definition, one has that $x_n \le \alpha_1 n^{x-1}$ and $y_n \le \alpha_2 n^{y-1} $ for all $n \gg 1$, for some $\alpha_1,\alpha_2>0$. Hence $(x_n + y_n) \le \alpha_1 n^{x-1} + \alpha_2 n^{y-1} \le 2\alpha n^{y-1}$ for all $n \gg 1$, where $\alpha := \max\{\alpha_1,\alpha_2\}$. It follows that $\cx(\{x_n + y_n\}) \le y = \sup\{x,y\}$.

   Again, if $\curv(\{x_n\})= \infty$ or $\curv(\{y_n\})= \infty$, then there is nothing to prove. So we may assume that $\curv(\{x_n\})= x$, $\curv(\{y_n\})= y$, and both are finite. Without loss of generality, one may assume that $x\le y$. Consider $\epsilon>0$. Since $\curv(\{x_n\})= x$ and $\curv(\{y_n\})= y$, by definition, one has that $x_n \le (x+\epsilon)^n$ and $y_n \le (y+\epsilon)^n$ for all $n \gg 1$. Hence $(x_n + y_n) \le (x+\epsilon)^n + (y+\epsilon)^n \le 2 (y+\epsilon)^n$ for all $n \gg 1$. It follows that $\limsup (x_n + y_n)^{1/n} \le (y+\epsilon)$. This holds for every $\epsilon > 0$. Therefore $\curv(\{x_n + y_n\}) \le y = \sup\{x,y\}$. This completes the proof.
\end{proof}

\begin{lemma}\label{lem:curv-poly-ieq}
    Let $\{a_n\}$ and $\{b_n\}$ be sequences of non-negative integers and $P(t)$ be a nonzero polynomial with non-negative integer coefficients such that $\sum_{n=0}^{\infty} a_n t^n = (\sum_{n=0}^{\infty} b_n t^n) P(1/t)$ holds as a formal Laurent series. Then $ \cx\{(b_n)\} \le \cx\{(a_n)\}$ and $ \curv\{(b_n)\} \le \curv\{(a_n)\}$.
\end{lemma}

\begin{proof}
    Consider $P(t)= \sum_{i=0}^r c_i t^i$. Comparing the coefficient of $t^n$ from both sides of the equality $\sum_{n=0}^{\infty} a_n t^n = (\sum_{n=0}^{\infty} b_n t^n) P(1/t)$, one obtains that $a_n=\sum_{i=0}^r b_{n+i} c_i$ for all $n \ge 0$. Therefore, since $\{a_n\}$, $\{b_n\}$ and $\{c_n\}$ are sequences of non-negative integers, it follows that $b_{n+i} c_i \le a_n$ for all $0\le i\le r$ and for all $n\ge 0$. Since $P(t)$ be a nonzero polynomial with non-negative integer coefficients, there exists $0 \le m \le r$ such that $c_m \ge 1$. Thus $b_{n+m}\le b_{n+m} c_m \le a_n$ for all $n\ge 0$. This implies that $ \cx\{(b_n)\} \le \cx\{(a_n)\}$ and $ \limsup (b_n)^{1/n} \le \limsup (a_n)^{1/n}$, i.e., $ \curv\{(b_n)\} \le \curv\{(a_n)\}$.
\end{proof}



\section{Complexity and curvature of (pair of) modules}\label{sec:cx-curv-mod}

In this section, we recall and introduce certain notions and properties of complexities and curvatures of pair of modules based on both Ext and Tor functors.

Let $M$ be an $R$-module. Let $\mu_R(M)$ and $\lambda_R(M)$ denote the minimal number of generators and the length of $M$ respectively. For every integer $n$, let $\beta_n^R(M) :=\mu(\Ext_R^n(M,k))$ and $\mu_R^n(M) := \mu(\Ext_R^n(k,M))$ be the $n$th Betti and Bass numbers of $M$ respectively. The type of $M$ is defined to be $\type_R(M):=\mu_R^t(M)$, where $t=\depth(M)$. When the base ring $R$ is clear from the context, we drop the subscript from $\mu_R(M)$, $\lambda_R(M)$ and $\type_R(M)$. Denote $(-)^\vee := \Hom_R(-,E)$, the Matlis dual, where $E$ is the injective hull of $k$.

The notion of complexity of a pair of modules was introduced by Avramov-Buchweitz.  


\begin{definition}\cite[pp.~286]{AB00}\label{defn:cx-M-N}
	The complexity of a pair of $R$-modules $M$ and $N$ is defined to be
	$$\cx_R(M,N):= \cx\big\{ \mu_R(\Ext_R^n(M,N)) \big\}_{n\ge 0} \quad \mbox{[see \Cref{defn:cx-curv-x_n}].}$$
\end{definition}

In a similar way, we define Ext-curvature of a pair of modules.

\begin{definition}\label{defn:curv-M-N}
	We define the Ext-curvature of a pair of $R$-modules $M$ and $N$ as
	$$\curv_R(M,N):= \curv\big\{ \mu_R(\Ext_R^n(M,N)) \big\}_{n\ge 0} \quad \mbox{[see \Cref{defn:cx-curv-x_n}].}$$
\end{definition}

One can also define the dual notions of (Ext-) complexity and curvature, which are Tor-complexity and Tor-curvature respectively. The Tor-complexity was first studied in \cite[pp.~4]{Dao}.

\begin{definition}\label{defn:tcx-tcurv}
	The Tor-complexity and Tor-curvature of a pair of $R$-modules $M$ and $N$ are defined as
	\begin{align*}
	\tcx_R(M,N) &:= \cx\big\{ \mu_R(\Tor^R_n(M,N)) \big\}_{n\ge 0} \quad \mbox{ and}\\
	\tcurv_R(M,N) &:= \curv\big\{ \mu_R(\Tor^R_n(M,N)) \big\}_{n\ge 0} \quad \mbox{[see \Cref{defn:cx-curv-x_n}].}
	\end{align*}
\end{definition}

In proving results on complexity and curvature, one often has to work with half-exact sequences of modules, where unfortunately $\mu(-)$ is incomparable along such sequences. Since length still behaves well along such sequences, it would be more desirable to have an expression of complexity and curvature in terms of $\lambda(-)$ in place of $\mu(-)$.

\begin{lemma}\label{cx-mu-l}
	Let $\{M_i\}_{i\ge 0}$ be a sequence of $R$-modules. Suppose there exists an integer $h\ge 0$ such that $\fm^h M_i =0$ for all $i \gg 0$. Then the following hold.
    \begin{enumerate}[\rm (1)]
        \item $\cx_R(\{ \mu(M_i)\})=\cx_R(\{ \lambda(M_i)\})=\cx_R(\{ \type(M_i)\})$.
        \item $\curv_R(\{ \mu(M_i)\})=\curv_R(\{\lambda(M_i)\})=\curv_R(\{\type(M_i)\})$.
    \end{enumerate}
\end{lemma}

\begin{proof}
    It is shown in \cite[Lem.~2.5]{DV09} that $\cx_R(\{ \mu(M_i)\})=\cx_R(\{ \lambda(M_i)\})$. By a similar argument, one obtains that $\curv_R(\{ \mu(M_i)\})=\curv_R(\{\lambda(M_i)\})$. For the other two equalities, set $N_i := (M_i)^{\vee}$ for all $i$. Then $\fm^h N_i =0$ for all $i \gg 0$. Hence $\cx_R(\{ \mu(N_i)\})=\cx_R(\{ \lambda(N_i)\})$ and $\curv_R(\{ \mu(N_i)\})=\curv_R(\{\lambda(N_i)\})$. For an $R$-module $M$, since $\mu(M^{\vee})=\type(M)$ and $\lambda(M^{\vee})=\lambda(M)$, the desired equalities follow.
\end{proof}

As an immediate consequence of Lemma~\ref{cx-mu-l}, one obtains  the following result. 

\begin{corollary}\label{cor:cx-curv-length}
	If $M$ or $N$ has finite length, then $\mu(-)$ can be replaced by $\lambda(-)$ in {\rm Definitions~\ref{defn:cx-M-N}, ~\ref{defn:curv-M-N} and \ref{defn:tcx-tcurv}.}
\end{corollary}

The next lemma produces more classes of modules for which $\mu(-)$ can be replaced by $\lambda(-)$ in {\rm Definitions~\ref{defn:cx-M-N}, ~\ref{defn:curv-M-N} and \ref{defn:tcx-tcurv}.}

\begin{lemma}\label{lem:Tor-m-zero}
    Suppose $M$ is an $R$-module locally of finite projective dimension on $\Spec(R) \smallsetminus \{ \fm \}$. Then there exist positive integers $g$ and $h$ such that:
    \begin{enumerate}[\rm (1)]
        \item $\fm^g \Ext_R^n(M,N)=0$ for all $n>\dim(R)$ and for all $R$-modules $N$, and
        \item $\fm^h \Tor_n^R(M,N)=0$ for all $n>\dim(R)$ and for all $R$-modules $N$.
    \end{enumerate}

Consequently, $\mu(-)$ can be replaced by $\lambda(-)$ in {\rm Definitions~\ref{defn:cx-M-N}, ~\ref{defn:curv-M-N} and \ref{defn:tcx-tcurv}.}    
\end{lemma}

\begin{proof}
Set $d:=\dim(R)$. Then $\dim(R_{\fp})\le d$ for all $\fp \in \Spec(R)$. Hence, $\depth(\Omega^d_R(M)_{\fp})\ge \depth(R_{\fp})$ for all $\fp \in \Spec(R)$ (cf.~\cite[1.3.7]{BH98}). Therefore, from the given hypothesis, by Auslander-Buchsbaum formula, one obtains that $\Omega^d_R(M)$ is locally free on $\Spec(R) \smallsetminus \{ \fm \}$. Note that $\Ext_R^n(M,N) \cong \Ext_R^{n-d}(\Omega^d_R(M),N)$ and $\Tor_n^R(M,N) \cong \Tor_{n-d}^R(\Omega^d_R(M),N)$ for all $n>d$. So (1) follows from \cite[Lem.~4.3]{DV09}. For (2), the proof is similar to that of \cite[Lem.~4.3]{DV09}. In fact, by using the functor $ (-) \otimes_R N $ instead of $\Hom_R(-,N)$ in the proof of \cite[Lem.~4.3]{DV09}, one obtains the desired result on Tor.

The last part follows from \Cref{cx-mu-l}.  
\end{proof}

The (projective and injective) complexity and curvature of a module were introduced by Avramov. See \cite[Sec.~1]{Avr96} for the exact references of these invariants.

\begin{definition}[Avramov]{\,}\label{defn:cx-curv}
	\begin{enumerate}[\rm (1)]
		\item The (projective) complexity and curvature of $M$ are defined to be $$\cx_R(M):=\cx_R(M,k) \quad\mbox{and}\quad \curv_R(M):=\curv_R(M,k).$$
		\item The injective complexity and curvature of $M$ are defined by
		$$\injcx_R(M):=\cx_R(k,M) \quad\mbox{and}\quad \injcurv_R(M):=\curv_R(k,M).$$
	\end{enumerate}
\end{definition}


The following remarks are well known, and these follow directly from the definitions of complexity and curvature. The statements (1) and (2) in \Cref{rmk:cx-curv} can be observed from \Cref{rmk:curv-zero}.

\begin{remark}\label{rmk:cx-curv}
	For $R$-modules $M$ and $N$, the following hold true.
	\begin{enumerate}[\rm (1)]
		\item \label{finite-pd-cx-curv}$\pd_R(M)<\infty$ $\Longleftrightarrow$ $\cx_R(M) = 0$ $\Longleftrightarrow$ $\curv_R(M) = 0$ $\Longleftrightarrow$ $\curv_R(M) < 1$.
		\item \label{finite-id-cx-curv}$\id_R(M)<\infty$ $\Longleftrightarrow$ $\injcx_R(M) = 0$ $\Longleftrightarrow$ $\injcurv_R(M) = 0$ $\Longleftrightarrow$ $\injcurv_R(M) < 1$.
		\item \label{finite-cx-implies-curv} It follows from \Cref{rmk:cx-curv-ineq}.(3) that
        \begin{enumerate}
            \item $\cx_R(M,N)<\infty \Longrightarrow\curv_R(M,N)\le 1$, and $\tcx_R(M,N)<\infty \Longrightarrow \tcurv_R(M,N)\le 1$.
            \item In particular, $\cx_R(M) < \infty $ $\Longrightarrow$ $\curv_R(M) \le 1$, and $\injcx_R(M) < \infty $ $\Longrightarrow$ $\injcurv_R(M) \le 1$.
        \end{enumerate}
		\item Since $\beta_n^R(k) = \mu_R^n(k)$ for all $n$, it follows that
		\begin{center}
			$\cx_R(k) = \injcx_R(k)$ \ and \ $\curv_R(k) = \injcurv_R(k)$.
		\end{center}
\item Since the minimal number of generators of a module is zero if and only if the module itself is zero, it follows from \Cref{rmk:curv-zero} and Definitions~\ref{defn:cx-M-N}, ~\ref{defn:curv-M-N} and \ref{defn:tcx-tcurv} that 
\begin{enumerate}
    \item $\cx_R(M,N)=0\Longleftrightarrow \curv_R(M,N)<1\Longleftrightarrow \Ext^n_R(M,N)=0 \text{ for all } n\gg 0 $.
    \item $\tcx_R(M,N)=0\Longleftrightarrow \tcurv_R(M,N)<1\Longleftrightarrow \Tor_n^R(M,N)=0 \text{ for all } n\gg 0$.
\end{enumerate}  
\end{enumerate}
\end{remark}

The residue field has extremal complexity and curvature. Moreover, complexity and curvature of the residue field characterize complete intersection local rings.

\begin{proposition}\label{prop:cx-curv-facts}{\ }
	\begin{enumerate}[\rm (1)]
		\item \label{cx-curv-M-k}\cite[Prop.~2]{Avr96} For every $R$-module $M$,
		\begin{enumerate}[\rm (a)]
			\item $\cx_R(M) \le \cx_R(k)$ and $\injcx_R(M) \le \injcx_R(k) =\cx_R(k)$.
			\item $\curv_R(M) \le \curv_R(k)$ and $\injcurv_R(M) \le \injcurv_R(k) = \curv_R(k) $.
		\end{enumerate}
		\item \label{char-CI-via-cx-curv}\cite[Thm.~3]{Avr96} The following statements are equivalent:
		\begin{enumerate}[\rm (a)]
			\item $R$ is complete intersection $($resp., of codimension $c$$)$.
			\item $\cx_R(k)<\infty$ $($resp., $\cx_R(k)=c<\infty$$)$.
			\item $\curv_R(k)\le 1$.
		\end{enumerate}
		\item \label{CI-ring-finite-cx-curv}It follows from $(1)$ and $(2)$ that if $R$ is complete intersection, then for every $R$-module $M$, $\cx_R(M) < \infty$, $\injcx_R(M) < \infty$, $\curv_R(M) \le 1$ and $\injcurv_R(M) \le 1$.
        \item \label{CI-ring-finite-cx-pair}
       In view of \cite[Thm.~II.(3)]{AB00}, if $R$ is complete intersection, then for every pair of $R$-modules $M$ and $N$, $\cx_R(M,N) = \cx_R(N,M) \le \min\{\cx_R(M),\cx_R(N)\}< \infty$ by \eqref{CI-ring-finite-cx-curv}.
	\end{enumerate}
\end{proposition}

\begin{remark}\label{rmk:flat-ext-CM}
Let $(R,\fm)\to (S,\fn)$ be a flat homomorphism such that $\fn=\fm S$. Let $M$ and $N$ be $R$-modules. Then, the following hold.
\begin{enumerate}[\rm(1)]
    \item $\depth_R(M)=\depth_S(M\otimes_R S)$ (see, e.g., \cite[1.2.16.(a)]{BH98}).
    \item $\dim_R(M)=\dim_S(M\otimes_R S)$ (cf.~\cite[A.11.(b)]{BH98}).
    \item \label{mh-tensor-S}$\fm^h M \otimes_R S \cong (\fm^h S)(M\otimes_R S)=\fn^h (M\otimes_R S)$ (see \cite[Lem.~5.2.5.(1)]{ddd}).
    \item $e_R(M)=e_S(M\otimes_R S)$ (see, e.g., \cite[4.3]{ddd}).
    \item $\mu_R(M)=\mu_S(M\otimes_R S)$ and $\type_R(N)=\type_S(N\otimes_R S)$.
    \item 
    $\cx_R(M,N)=\cx_S(M\otimes_R S, N\otimes_R S)$, and $\curv_R(M,N)=\curv_S(M\otimes_R S, N\otimes_R S)$.
    \item 
    $\tcx_R(M,N)=\tcx_S(M\otimes_R S, N\otimes_R S)$, and $\tcurv_R(M,N)=\tcurv_S(M\otimes_R S, N\otimes_R S)$.
    \item 
    For each fixed $i$ and $h$, if $\fm^h\Ext^i_R(M,N)=0$, then $\fn^h\Ext^i_S(M\otimes_R S, N\otimes_R S)=0$.
    \item 
    For each fixed $i$ and $h$, if $\fm^h\Tor_i^R(M,N)=0$, then $\fn^h\Tor_i^S(M\otimes_R S, N\otimes_R S)=0$.
\end{enumerate}
\end{remark}

\begin{proof}
    Since $R\to S$ is faithfully flat, for each $i$, note that
    \begin{equation}\label{Ext-Tor-R-S}
        \Ext^i_R(M,N)\otimes_R S\cong \Ext^i_S(M\otimes_R S, N\otimes_R S) \mbox{ and } \Tor_i^R(M,N)\otimes_R S\cong \Tor_i^S(M\otimes_R S, N\otimes_R S)
    \end{equation}    
    It follows from tensoring a minimal $R$-free resolution of $M$ by $S$ that $\mu_R(M)=\mu_S(M\otimes_R S) $. Set $t:=\depth_R(N)$. Then $t=\depth_S(N\otimes_R S)$. Tensoring $\Ext^t_R(R/\fm,N)$ by $S$, and using \eqref{Ext-Tor-R-S}, one gets that $\type_R(N)=\type_S(N\otimes_R S)$. This proves (5).
    The desired equalities in (6) and (7) follow from \eqref{Ext-Tor-R-S} and the fact that $\mu_R(X)=\mu_S(X\otimes_R S) $ for any $R$-module $X$. The statements in (8) and (9) can be derived from \eqref{Ext-Tor-R-S} and \eqref{mh-tensor-S}.
\end{proof}

In the next three lemmas, we record how complexity and curvature of (pair of) modules behave with respect to going modulo elements regular on the module.

\begin{lemma}\label{lem:cx-curv-M-M/xM}
    Let $x$ be an $M$-regular element. Then:
    \begin{enumerate}[\rm (1)]
        \item $\cx_R(M) = \cx_R(M/xM)$ and $\curv_R(M) = \curv_R(M/xM)$.
        \item $\injcx_R(M)=\injcx_R(M/xM)$ and $\injcurv_R(M)=\injcurv_R(M/xM)$.
    \end{enumerate}
\end{lemma}

\begin{proof}
(1) Considering the long exact sequence
\begin{equation*}
	\cdots \to \Ext_R^{n-1}(M,k) \stackrel{x} \longrightarrow \Ext_R^{n-1}(M,k) \to \Ext_R^{n}(M/xM,k) \to \Ext_R^{n}(M,k)\stackrel{x} \longrightarrow \cdots,
\end{equation*}
 and remembering that multiplication by $x$ on $\Ext^i_R(M,k)$ is the zero map, one obtains that $\beta_n^R(M/xM) = \beta_n^R(M)+\beta_{n-1}^R(M)$ for all $n$. Now we are done by \cite[Prop.~2.2.(3)]{DV09} and Lemma~\ref{lem:curv-sum} respectively. 

(2) In this case, considering the long exact sequence
\begin{equation*}
	\cdots \to \Ext_R^{n}(k,M) \stackrel{x} \longrightarrow \Ext_R^{n}(k,M) \longrightarrow \Ext_R^{n}(k,M/xM) \longrightarrow \Ext_R^{n+1}(k,M)\stackrel{x} \longrightarrow \cdots,
\end{equation*}
one gets that $ \mu_R^n(M/xM) = \mu_R^n(M) + \mu_R^{n-1}(M) $ for all $n$. Hence the proof follows as (1).
\end{proof}

Let $N\subseteq M$ be $R$-modules such that $IN=0$ and $I(M/N)=0$ for some ideal $I$ of $R$. Then $I^2M\subseteq IN=0$, i.e., $I^2M=0$. We use this observation in the following two lemmas.

\begin{lemma}\label{lem:cx-curv-M-M/xM-Tor}
    Let $X$ and $M$ be two $R$-modules such that there exists a positive integer $h$ satisfying $\fm^h \big(\Tor_n^R(X,M) \big)=0$ for all $n \gg 0$. Let $x$ be an $M$-regular element. Then
    \begin{enumerate}[\rm (1)]
        \item $\fm^{2h}\Tor_n^R(X,M/xM) \big)=0$ for all $n \gg 0$.
        \item $\tcx_R(X,M/xM)\le \tcx_R(X,M)$ and $\tcurv_R(X,M/xM) \le \tcurv_R(X,M)$.
    \end{enumerate}
    In general, for an $M$-regular sequence $\mathbf x$, one has that  $$\tcx_R(X,M/\mathbf x M)\le \tcx_R(X,M) \; \mbox{ and } \; \tcurv_R(X,M/{\mathbf x}M) \le \tcurv_R(X,M).$$
\end{lemma}

\begin{proof}
    The short exact sequence $0 \to M \stackrel{x} \longrightarrow M \to M/x M \to 0$ induces an exact sequence $\Tor_{n}^R(X,M) \to \Tor_n^R(X,M/xM) \to \Tor_{n-1}^R(X,M)$, which yields, for all $n\gg 0$, that $\fm^{2h}\Tor_n^R(X,M/xM) \big)=0$  and $\lambda \big(\Tor_n^R(X,M/xM) \big)\le \lambda \big(\Tor_n^R(X,M) \big)+\lambda \big(\Tor_{n-1}^R(X,M) \big)$. So, by Lemmas~\ref{cx-mu-l} and \ref{lem:cx-curv-seq}, one gets the desired inequalities. This proves (1) and (2).
    
    The general case follows by repeatedly using (1) and (2).
\end{proof}

The counterpart of \Cref{lem:cx-curv-M-M/xM-Tor} for Ext modules is the following.

\begin{lemma}\label{lem:cx-curv-M-M/xM-Ext}
    Let $X$ and $M$ be two $R$-modules such that there exists a positive integer $h$ satisfying $\fm^h \big(\Ext_R^n(X,M)\big)=0$ for all $n \gg 0$. 
    \begin{enumerate}[\rm (1)]
        \item Let $x$ be an $M$-regular element. Then $\fm^{2h}\Ext^n_R(X,M/xM) \big)=0$ for all $n \gg 0$. Moreover,
        \[
            \cx_R(X,M/xM)\le \cx_R(X,M) \; \mbox{ and } \; \curv_R(X,M/xM)\le \curv_R(X,M).
        \]
        \item 
        Let $y$ be an $X$-regular element. Then $\fm^{2h}\Ext^n_R(X/yX,M) \big)=0$ for all $n \gg 0$. Moreover
        \[
            \cx_R(X/yX,M)\le \cx_R(X,M) \; \mbox{ and } \; \curv_R(X/yX,M)\le \curv_R(X,M).
        \]
        \item 
        If $\mathbf y$ is an $X$-regular sequence and $\mathbf x$ is an $M$-regular sequence, then
        \[
        \cx_R(X/{\mathbf y} X, M/{\mathbf x}M)\le \cx_R(X,M) \; \mbox{ and } \; \curv_R(X/{\mathbf y} X, M/{\mathbf x}M)\le \curv_R(X,M).
        \]
    \end{enumerate}
\end{lemma}

\begin{proof}
    (1) The short exact sequence $0 \to M \stackrel{x} \longrightarrow M \to M/x M \to 0$ induces an exact sequence $\Ext^n_R(X,M) \to \Ext^n_R(X,M/xM) \to \Ext^{n+1}_R(X,M)$, which yields, for all $n\gg 0$, that
    $\fm^{2h}\Ext^n_R(X,M/xM) \big)=0$ and $\lambda \big(\Ext^n_R(X,M/xM) \big)\le \lambda \big(\Ext^n_R(X,M) \big)+\lambda \big(\Ext^{n+1}_R(X,M) \big)$. Hence, by Lemmas~\ref{cx-mu-l} and \ref{lem:cx-curv-seq}, one obtains the desired inequalities.

    (2) The short exact sequence $0 \to X \stackrel{y} \longrightarrow X \to X/y X \to 0$ induces an exact sequence $\Ext^n_R(X,M) \to \Ext^{n+1}_R(X/yX,M) \to \Ext^{n+1}_R(X,M)$, which yields, for all $n\gg 0$, that $\fm^{2h}\Ext^n_R(X/yX,M) \big)=0$  and $\lambda \big(\Ext^{n+1}_R(X/yX,M) \big)\le \lambda \big(\Ext^n_R(X,M) \big)+\lambda \big(\Ext^{n+1}_R(X,M) \big)$. Hence the desired inequalities follow from Lemmas~\ref{cx-mu-l} and \ref{lem:cx-curv-seq}.

    (3) By using (1) and (2) one by one repeatedly, one gets the desired inequalities.    
\end{proof}

Our next lemma is a generalization of \cite[2.4]{ACST}.  

\begin{lemma}\label{lem:Ext-Tor-duality}
    Let $R$ be a CM local ring of dimension $d$ with a canonical module $\omega$. Let $L$ and $M$ be $R$-modules such that $M$ is CM of dimension $t$, and $\Tor_i^R(L,M)$ has finite length for all $i\ge 1$.
    Then
    \[
    \Ext_R^{t+i}(L,M^{\dagger}) \cong \Ext_R^d(\Tor_i^R(L,M),\omega) \;\mbox{ for all }i\ge 1,
    \]
    where $M^{\dagger} := \Ext_R^{d-t}(M,\omega)$.
\end{lemma}

\begin{proof}
    Consider an $R$-module $N$. We use the spectral sequences induced by tensor-hom adjunction
    \begin{equation*}
	   \mathbb{X} := \Hom_R(\mathbb{P}_L, \Hom_R(M,\mathbb{I}_N)) \cong \Hom_R(\mathbb{P}_L \otimes_R M, \mathbb{I}_N) =: \mathbb{Y}
	\end{equation*}
    as discussed in \cite[3.6]{GP24}. These double complexes induce the following spectral sequences:
    \begin{align}
        {}^vE_2^{p,q}(\mathbb{X}) &= \Ext_R^p\left(L, \Ext_R^q(M,N)\right) \mbox{ for all } p,q, \mbox{ and}\label{spec-seq-X}\\
        {}^hE_2^{p,q}(\mathbb{Y}) &= \Ext_R^q\left( \Tor_p^R(L,M), N \right) \mbox{ for all } p,q.\label{spec-seq-Y}
    \end{align}
    Since $\mathbb{X}\cong \mathbb{Y}$, the cohomologies of their total complexes are also isomorphic, i.e.,
    \begin{equation}\label{coh-Tot-X-Y-iso}
	   H^n(\Tot(\mathbb{X})) \cong H^n(\Tot(\mathbb{Y})) \; \mbox{ for all } n.
	\end{equation}
    Set $N:=\omega$. Then, since $M$ is CM of dimension $t$, one has that $\Ext_R^q(M,\omega) = 0$ for all $q\neq d-t$. 
    So the spectral sequence ${}^vE_2^{p,q}(\mathbb{X})$	collapses on the line $q=d-t$. Hence
    \begin{equation}\label{coh-Tot-X}
	H^n(\Tot(\mathbb{X})) \cong {}^vE_2^{n-d+t,d-t}(\mathbb{X}) = \Ext_R^{n-d+t}(L, \Ext_R^{d-t}(M,\omega)) = \Ext_R^{n-d+t}(L, M^{\dagger}) \; \mbox{ for all } n,
	\end{equation}
    see, e.g., \cite[5.2.7, p.~124]{We94}. On the other hand, since $\lambda(\Tor_i^R(L,M))<\infty$ for all $i\ge 1$,
    \begin{equation}\label{hE2-collapses}
        {}^hE_2^{p,q}(\mathbb{Y}) = \Ext_R^q( \Tor_p^R(L,M), \omega)=0 \mbox{ for all } p\ge 1 \mbox{ and } q\neq d, \mbox{ or for } p=0 \mbox{ and } q>d.
    \end{equation}
    Note that the differentials in ${}^hE_2^{p,q}$ are given by ${}^h d_2^{p,q} :{}^hE_2^{p,q} \longrightarrow {}^hE_2^{p-1,q+2}$. So \eqref{hE2-collapses} yields that ${}^hE_{\infty}^{p,q}(\mathbb{Y}) = {}^hE_2^{p,q}(\mathbb{Y})$ for all $p\ge 1$ or $q>d$. Moreover, ${}^hE_{\infty}^{p,q}(\mathbb{Y}) = 0$ for all $p\ge 1$ and $q\neq d$. Also ${}^hE_{\infty}^{p,q}(\mathbb{Y}) = 0$ for all $q>d$. Thus, for all $n\ge d+1$, ${}^hE_{\infty}^{n-d,d}={}^hE_2^{n-d,d}$ is the only possible nonzero term on the line $p+q=n$ in the abutment of the spectral sequence ${}^hE_r^{p,q}(\mathbb{Y})$. Therefore 
    \begin{equation}\label{coh-Tot-Y}
	H^n(\Tot(\mathbb{Y})) \cong {}^hE_2^{n-d,d}(\mathbb{Y}) = \Ext_R^d(\Tor_{n-d}^R(L,M), \omega) \; \mbox{ for all } n-d \ge 1.
	\end{equation}
    Hence the desired isomorphisms can be observed from \eqref{coh-Tot-X-Y-iso}, \eqref{coh-Tot-X} and \eqref{coh-Tot-Y}.
\end{proof}

Finally, we establish a duality between Ext and Tor complexities and curvatures of pairs of modules. 

\begin{proposition}\label{cx-duality}
Let $R$ be CM of dimension $d$ admitting a canonical module $\omega$.
\begin{enumerate}[\rm (1)]
    \item 
    Let $\{X_i\}$ and $\{Y_i\}$ be  sequences of $R$-modules such that $\fm^h Y_i=0$ for some $h\ge 0$ and $X_i\cong \Ext_R^d(Y_i, \omega)$ for all $i\gg 0$.  Then, $\cx(\{\mu(X_i)\})=\cx(\{\mu(Y_i)\})$ and $\curv(\{\mu(X_i)\})=\curv(\{\mu(Y_i)\})$.
    \item 
    Let $M$ be a CM $R$-module. Set $M^{\dagger} := \Ext_R^{d-\dim(M)}(M,\omega)$. Then, the following hold.
    \begin{enumerate}[\rm (a)]
        \item If $L$ is an $R$-module such that $\fm^h\Tor^R_{\gg 0}(L,M)=0$ for some $h\ge 0$, then $\fm^h \Ext_R^{\gg 0}(L, M^{\dagger})=0$, $\cx_R(L,M^{\dagger})=\tcx_R(L,M)$, and $\curv_R(L,M^{\dagger})=\tcurv_R(L,M)$.
        \item In particular, $\injcx_R(M^{\dagger})=\cx_R(M)$ and $\injcx_R(M)=\cx_R(M^{\dagger})$.
        \item $\injcurv_R(M^{\dagger})=\curv_R(M)$ and $\injcurv_R(M)=\curv_R(M^{\dagger})$.
    \end{enumerate} 
\end{enumerate}
\end{proposition}

\begin{proof}
(1) Since $\fm^h Y_i=0$ for all $i\gg 0$, it follows that $\fm^h X_i=0$ for all $i\gg 0$. Moreover, $X_i^{\vee} \cong \Ext^d_R(Y_i,\omega)^{\vee} \cong H^0_{\fm}(Y_i)\cong Y_i$ for all $i\gg 0$ (cf.~\cite[3.5.9]{BH98} for the second isomorphism).  Thus, in view of \Cref{cx-mu-l}.(1), one gets that $\cx(\{\mu(X_i)\})=\cx(\{\lambda(X_i)\})=\cx(\{\lambda(X_i^{\vee})\}) = \cx(\{\lambda(Y_i)\})=\cx(\{\mu(Y_i)\})$. Similarly, \Cref{cx-mu-l}.(2) yields that $\curv(\{\mu(X_i)\})=\curv(\{\mu(Y_i)\})$.

(2) Since $\cx_R(\Omega_R(-), X)=\cx_R((-),X)$ and $\tcx_R(\Omega_R(-), X)=\tcx_R((-),X)$, we may pass to a high enough syzygy of $L$, and assume that $\fm^h\Tor^R_{\ge 1}(L,M)=0$. Then, by \Cref{lem:Ext-Tor-duality}, $\Ext_R^{t+i}(L, M^{\dagger}) \cong \Ext^d_R(\Tor^R_i(L,M), \omega)$ for all $i\ge 1$, where $t=\dim M$. It follows that $\fm^h\Ext_R^{i\gg 0}(L, M^{\dagger})=0$. Moreover, (1) yields that $\cx_R(L,M^{\dagger}) = \tcx_R(L,M)$ and $\curv_R(L,M^{\dagger})=\tcurv_R(L,M)$. Putting $L=k$ in (a), one obtains that $\injcx_R(M^{\dagger})=\cx_R(M)$ and $\injcurv_R(M^{\dagger})=\curv_R(M)$. As $M^{\dagger}$ is also CM, and $M^{\dagger\dagger}\cong M$ (see, e.g., \cite[3.3.10]{BH98}), replacing $M$ by $M^{\dagger}$, the other two equalities in (b) and (c) follow.
\end{proof}

\section{On multiplicities of Cohen-Macaulay modules}\label{sec:mult}

Here, we establish two inequalities involving multiplicity of a CM module, see \Cref{lem:mult-inequality}.(1) and \Cref{thm:e-mu-type}. The first one can be derived from \cite[Thm.~14]{Pu03}. However, we give a different and short proof of this result by mimicking the argument of Abhyankar's inequality. Along with the second inequality, \Cref{thm:e-mu-type} provides a new characterization of Ulrich modules (\Cref{defn:Ulrich-mod}) in terms of type of a module. Furthermore, in this section, we discuss some properties of Ulrich modules and CM modules of minimal multiplicity. A connection between these two notions is given in \Cref{prop:cm-min-ul}.

For the definition of (Hilbert-Samuel) multiplicity of a module $M$ with respect to an $\fm$-primary ideal $I$, denoted by $e(I,M)$, we refer the reader to \cite[4.6.1 and 4.6.2]{BH98}. We simply denote $e(\fm,M)$ by $e(M)$.

\begin{lemma}\label{lem:mult-inequality}
    Let $M$ be a CM $R$-module of dimension $r$. Then the following hold.
    \begin{enumerate}[\rm (1)]
        \item $e(M) \ge \mu(\fm M) + (1-r) \mu(M)$.
        \item If $\fm^2M = ({\bf x}) \fm M$ for some system of parameters ${\bf x}$ of $M$, then
        \[
            e(M) = \lambda(M/({\bf x})M) = \mu(\fm M) + (1-r) \mu(M).
        \]
        \item If the residue field $k$ is infinite, then there exists a system of parameters ${\bf x}$ of $M$ which is a reduction of $\fm$ with respect to $M$. Moreover, if $e(M) = \mu(\fm M) + (1-r) \mu(M)$ holds, then $\fm^2M = ({\bf x}) \fm M$ holds for every such reduction ${\bf x}$.
    \end{enumerate}
\end{lemma}  

\begin{proof}
	Let ${\bf x} := x_1,\ldots,x_r $ be a system of parameters of $M$. Since $M$ is CM, then ${\bf x}$ is $M$-regular, and $e(({\bf x}),M)=\lambda(M/({\bf x})M)$ (see, e.g., \cite[Rmk.~2.4]{DG23}). Consider the chains
    \[
    	({\bf x})\fm M \subseteq \fm^2 M \subseteq \fm M \subseteq M \quad \mbox{and} \quad ({\bf x})\fm M \subseteq ({\bf x})M \subseteq \fm M \subseteq M
    \]
    of submodules of $M$. From these chains, one obtains that    \begin{align}\label{lambda-equals}
    	e(({\bf x}),M)&=\lambda(M/({\bf x})M) = \lambda(M/\fm M) + \lambda(\fm M/({\bf x})M) \\ 
        &= \mu(M)+ \lambda\big(\fm M/({\bf x})\fm M\big) -\lambda \big( ({\bf x})M/({\bf x})\fm M\big) \nonumber\\
    	&= \mu(M)+\lambda\big(\fm M/\fm^2 M\big) + \lambda\big(\fm^2 M/({\bf x})\fm M\big) - \lambda\big( ({\bf x})M/({\bf x})\fm M \big) \nonumber\\
        &=\mu(M)+ \mu(\fm M) + \lambda\big(\fm^2 M/({\bf x})\fm M\big) - \mu\big(({\bf x})M\big).\nonumber
    \end{align}
    Note that ${\bf x}$ is $M$-regular. So $ \Tor_1^R(R/({\bf x}), M) = 0 $, cf.~\cite[1.1.12.(a)]{BH98}. This implies that $({\bf x})M \cong ({\bf x})\otimes_R M$, and hence $\mu\big(({\bf x})M\big) = \mu\big(({\bf x})\big) \mu(M) = r \mu(M)$. Therefore \eqref{lambda-equals} yields that    \begin{align}\label{lambda-inequality}
        e(({\bf x}),M) =\lambda(M/({\bf x})M) &= \mu(M)+ \mu(\fm M) + \lambda\big(\fm^2 M/({\bf x})\fm M\big) - r \mu(M) \\
        &\ge \mu(M)+ \mu(\fm M) - r \mu(M).\nonumber
    \end{align}
    
    (1) In order to prove the inequality, due to Remark \ref{rmk:flat-ext-CM}, we may pass to the faithfully flat extension $R[X]_{\fm[X]}$, and assume that the residue field $k$ is infinite. By \cite[Cor.~4.6.10]{BH98}, there exists a system of parameters $ {\bf x} $ of $M$ such that $e(M)=e(({\bf x}),M)$. Hence the inequality follows from \eqref{lambda-inequality}.

    (2) Let $\fm^2M = ({\bf x}) \fm M$ for some system of parameters ${\bf x}$ of $M$. Then, for the same system of parameters, the equalities in \eqref{lambda-inequality} yield that   \begin{equation}\label{lambda-equality}
        e(({\bf x}),M) =\lambda(M/({\bf x})M) = \mu(M)+ \mu(\fm M) - r \mu(M) \le e(M),
    \end{equation}
    where the last inequality is established in (1). Since $({\bf x}) \subseteq \fm $, by the definition of multiplicity, one derives that $e(M) = e(\fm,M) \le e(({\bf x}),M)$. Therefore the desired equalities in (2) follow from \eqref{lambda-equality}.
    
    (3) Since $k$ is infinite, by \cite[Cor.~4.6.8]{BH98}, there exists a system of parameters of $M$ which is a reduction of $\fm$ with respect to $M$. Let ${\bf x}$ be such a system of parameters of $M$. Then, since ${\bf x}$ is a reduction of $\fm$ with respect to $M$, in view of \cite[Cor.~4.6.5]{BH98}, $e(M)=e(({\bf x}),M)$. Therefore, since $e(M) = \mu(\fm M) + (1-r) \mu(M)$, it follows from the equalities in \eqref{lambda-inequality} that $\fm^2M = ({\bf x}) \fm M$.
\end{proof}


\begin{definition}\cite[Defn.~15]{Pu03}\label{defn:min-mult}
	An $R$-module $M$ is said to have minimal multiplicity if it is CM and $e(M)=\mu(\fm M)+(1-r) \mu(M)$, where $r= \dim(M)$.
\end{definition}

The following gives a recipe for constructing new modules of minimal multiplicity from old ones.

\begin{proposition}\label{prop:direct-sum-min-mult}
Let $M$ and $N$ be $R$-modules, and $X = M\oplus N$. Then, the following hold.
\begin{enumerate}[\rm(1)]
    \item
    If $M$ and $N$ have minimal multiplicity and the same dimension, then $X$ also has minimal multiplicity.
    \item
    If $X$ has minimal multiplicity, then so does $M$ and $N$. 
\end{enumerate}
\end{proposition}

\begin{proof}
(1) Note that $\dim(X)=\dim(M)=\dim(N)$. Thus $e(X)=e(M)+e(N)$. Hence, since $\fm (M\oplus N)\cong \fm M\oplus \fm N$, and the function $\mu(-)$ is additive, the assertion follows from \Cref{defn:min-mult}.

(2) Since $X$ is CM by assumption, and $X = M\oplus N$, it follows that $M$ and $N$ are CM of the same dimension as that of $X$. Let $r$ be this common dimension. Then
\begin{align}\label{min-mult-ineq}
    e(X) = e(M)+e(N) &\ge \{\mu(\fm M)+(1-r) \mu(M)\} + \{\mu(\fm N)+(1-r) \mu(N)\} \quad \mbox{[by \ref{lem:mult-inequality}.(1)]}\\
    &= \mu(\fm M \oplus \fm N) + (1-r) \ \mu(M\oplus N) = e(X).\nonumber
\end{align}
Clearly, the inequality in \eqref{min-mult-ineq} must be an equality, and one concludes that both $M$ and $N$ have minimal multiplicity.
\end{proof}

\begin{remark}\label{rmk:defn-min-mult}
    When the residue field $k$ is infinite, in view of Lemma~\ref{lem:mult-inequality}(2) and (3), a CM $R$-module $M$ has minimal multiplicity if and only if $\fm^2M = ({\bf x}) \fm M$ for some system of parameters ${\bf x}$ of $M$.
\end{remark}

Now, for a CM module $M$, we mainly prove that $e(M)\ge \type(M)$. Moreover, inequality becomes equality exactly when the module is Ulrich (in the sense of \cite[Defn.~2.1]{GTT15} and \cite{BHU}).

\begin{theorem}\label{thm:e-mu-type}
Let $M$ be a CM $R$-module. Then $e(M)\ge \max\{\mu(M),\type(M)\}$. Moreover, $M$ is Ulrich $($i.e., $e(M)=\mu(M)$$)$ if and only if $e(M)=\type(M)$.
\end{theorem}

\begin{proof}
    Set $r:=\dim(M)$. The inequality $e(M)\ge \mu(M)$ was shown in \cite[Rmk.~2.4]{DG23}. For the other inequality, first notice that if $r=0$, then $e(M)=\lambda(M)=\lambda(M^{\vee})\ge \mu(M^{\vee})=\lambda(k\otimes_R M^{\vee})=\lambda((k\otimes_R M^{\vee})^{\vee})=\lambda (\Hom_R(k,M))=\type(M)$. Now assume that $r>0$. By \Cref{rmk:flat-ext-CM}, we may pass to a faithfully flat extension, and assume that $R$ has infinite residue field. As in the proof of \cite[Rmk~2.4]{DG23}, one has that $e(M)=\lambda(M/\mathbf x M)$ for some $M$-regular sequence $\mathbf x=x_1,\dots,x_r$. Since $\dim(M/\mathbf x M)=0$, it follows that $\lambda(M/\mathbf x M)\ge \lambda(\Hom_R(k, M/\mathbf x M))=\lambda(\Ext^r_R(k,M))=\type(M)$, where the second last equality is obtained by \cite[1.2.4]{BH98}. Thus $e(M)\ge \type(M)$, which finishes the proof of the inequality.
    
    For the second part, assume that $e(M)=\type(M)$. Then, by the proof of the first part, there exists an $M$-regular sequence $\mathbf x$ such that $\lambda(M/\mathbf xM)=\lambda(\Hom_R(k,M/\mathbf x M))$. Hence, since $\Hom_R(k, M/\mathbf x M)$ embeds in $M/\mathbf x M$, it follows that $M/\mathbf x M\cong \Hom_R(k, M/\mathbf x M)$, which is a $k$-vector space. So $M/\mathbf x M$ is Ulrich (\cite[Prop.~2.2.(1)]{GTT15}). Thus $M$ is Ulrich by \cite[Prop.~2.2.(5)]{GTT15}. Conversely, assume that $M$ is Ulrich. Then there exists an $M$-regular sequence $\mathbf x=x_1,\dots,x_r$ such that $M/\mathbf x M$ is a $k$-vector space by \cite[Prop.~2.2.(1) and (4)]{GTT15}. Thus $M/\mathbf x M \cong k^{\oplus \mu(M/\mathbf x M)}\cong k^{\oplus \mu(M)}$. So $\Hom_R(k, M/\mathbf x M)\cong k^{\oplus \mu(M)}$. Therefore  $e(M)=\mu(M)=\lambda(\Hom_R(k,M/\mathbf x M))=\lambda(\Ext^r_R(k,M))=\type(M)$.
\end{proof}

\begin{definition}\cite[Defn.~2.1]{GTT15}, \cite{BHU}\label{defn:Ulrich-mod}
    A nonzero $R$-module $M$ is called Ulrich if it is CM and $e(M) = \mu(M)$.
\end{definition}

\begin{remark}\label{rmk:flat-ext-Ulrich}
    Let $(R,\fm)\to (S,\fn)$ be a local flat homomorphism such that $\fn=\fm S$. Then, as consequences of Remark~\ref{rmk:flat-ext-CM}, the following hold.
    \begin{enumerate}[\rm(1)]
        \item
        $M$ is an $R$-module of minimal multiplicity if and only if $M\otimes_R S$ is an $S$-module of minimal multiplicity.
        \item
        $M$ is an Ulrich $R$-module if and only if $M\otimes_R S$ is an Ulrich $S$-module. 
    \end{enumerate}
\end{remark}

\begin{remark}\label{rmk:Ulrich-min-mult}
Any Ulrich $R$-module has minimal multiplicity. Indeed, by Remark~\ref{rmk:flat-ext-Ulrich}, we may pass to the faithfully flat extension $R[X]_{\fm[X]}$, and assume that the residue field is infinite. So, if $M$ is an Ulrich $R$-module, then in view of \cite[Prop.~2.2.(2)]{GTT15}, $\fm M=({\bf x})M$ for some system of parameters ${\bf x}$ of $M$, and hence $M$ has minimal multiplicity by Lemma~\ref{lem:mult-inequality}.(2).
\end{remark}


\begin{definition}\label{defn:ring-min-mult}
    The ring $R$ is said to have minimal multiplicity if $R$ has minimal multiplicity as an $R$-module, i.e., $R$ is CM and $e(R)=\mu(\fm)-d+1$, where $d=\dim(R)$.
\end{definition}

\begin{remark}\label{rmk:min-mult-R-M}
    If $R$ has minimal multiplicity, then any MCM $R$-module $M$ has minimal multiplicity. Indeed, we may assume that the residue field is infinite. Then, there exists a maximal $R$-regular sequence $\mathbf x$ such that $\fm^2=(\mathbf x)\fm$, and hence $\fm^2M=(\mathbf x)\fm M$. Since $M$ is MCM, $\mathbf x$ is also a maximal $M$-regular sequence. Thus $M$ has minimal multiplicity by Lemma~\ref{lem:mult-inequality}.(2). 
\end{remark} 

If $R$ has minimal multiplicity, then the syzygy in a minimal free resolution of any given MCM $R$-module is Ulrich, \cite[Prop.~3.6]{KT18}. This fact is immensely useful and one of the driving forces behind many results on Ulrich modules, see, e.g., \cite{KT18}, \cite[Sec.~4]{DG23} and \cite[Sec.~5]{DK22}. We close the present section by a proposition which, in view of \Cref{rmk:min-mult-R-M}, highly generalizes this fact \cite[Prop.~3.6]{KT18}. 

\begin{proposition}\label{prop:cm-min-ul}
Let $L$, $M$ and $N$ be $R$-modules such that $M$ and $N$ are CM of same dimension, $M$ has minimal multiplicity, $\mu(M)=\mu(N)$, and there exists an exact sequence $0\to L \to M \to N \to 0$. Then, either $L=0$, or $L$ is Ulrich.
\end{proposition}

\begin{proof}
Due to Remarks~\ref{rmk:flat-ext-CM} and \ref{rmk:flat-ext-Ulrich}, we may assume, without loss of generality, that $R$ has infinite residue field. Denote $r:=\dim(M)$. We first prove that $L$ is CM of dimension $r$. Indeed, by depth lemma, $\depth(L)\ge r$. But also, $\dim(L)\le \dim(M)=r$. Thus, $L$ is CM of dimension $r$. Since $R$ has infinite residue field, by \cite[4.6.10]{BH98}, there exists  $\mathbf x:=x_1,\ldots,x_r$ in $\fm$ which is a reduction of $\fm$ with respect to $M\oplus N$. Then, $\mathbf x$ is a reduction of $\fm$ with respect to $M$ and $N$ individually. In particular, $\mathbf x$ is a system of parameters on both $M$ and $N$, and hence regular on both of them (cf.~\cite[2.1.2.(d)]{BH98}).  Thus, $\Tor^R_1(N, R/(\mathbf x))=0$ by \cite[1.1.12.(a)]{BH98}. So, tensoring the exact sequence $0\to L \to M \to N \to 0$ by $R/(\mathbf x)$, one obtains another exact sequence $0\to L/(\mathbf x)L \xrightarrow{g} M/(\mathbf x)M \to N/(\mathbf x)N \to 0.$  Note that $\mu(M/(\mathbf x)M)=\mu(M)=\mu(N)=\mu(N/(\mathbf x)N)$. We now claim that this implies $\Image(g) \subseteq \fm(M/(\mathbf x)M)$. Indeed, set $A:=\Image(g)$, $B :=M/(\mathbf x)M$ and $C :=N/(\mathbf x)N$. Then, $B/A\cong C$. Also, $\mu(B)=\mu(C)$ implies $B/\fm B\cong C/\fm C$. But $C/\fm C\cong \dfrac{(B/A)}{\fm(B/A)}\cong \dfrac{(B/A)}{(\fm B+A)/A}\cong B/(\fm B +A)$. Thus, $B/\fm B\cong B/(\fm B +A)$. Hence, $\fm B=\fm B+A$, i.e., $A\subseteq \fm B$. Since $\mathbf x$ is a reduction of $\fm$ with respect to $M$, \Cref{lem:mult-inequality}.(3) yields that $\fm^2M=(\mathbf x)\fm M$. So $\fm A \subseteq \fm^2B = \fm^2(M/(\mathbf x)M)=0$. Since $A\cong L/(\mathbf x)L$, it follows that $\fm L=(\mathbf x)L$. Thus, $L$ is Ulrich by \cite[Prop.~2.2.(2)]{GTT15}.  
\end{proof}




\section{On complexity and curvature of (pair of) some CM modules}\label{sec:cx-curv-cm-mod}

In this section, we show that CM modules $M$ are not far from having maximal projective or injective complexity and curvature, provided $e(M)\le 2 \mu(M)$ or $ e(M) \le 2 \type(M)$ holds respectively. Our main results of this section, \Cref{thm:main-cx-curv-CM-mod} and \Cref{cor:mcm-aprrox}, are far more general, and establish upper bounds for complexity and curvature of arbitrary modules $X$ in terms of complexities and curvatures of pair of modules $X$ and $M$. We start with the following lemma.


\begin{lemma}\label{lem:M-fin-len-ineq}
	Let $M$ be an $R$-module of finite length, and $X$ be an $R$-module. Then
	\begin{enumerate}[\rm (1)]
        \item 
		$\beta_{n+1}^R(X) \mu(M) \le \big(\lambda(M)-\mu(M)\big) \beta_n^R(X) + \lambda\big(\Tor_{n+1}^R(X,M)\big)$ for all $n\ge 0$.
		\item 
		$\beta_{n+1}^R(X) \type(M) \le \big(\lambda(M)-\type(M)\big) \beta_n^R(X) + \lambda\big(\Ext_R^{n+1}(X,M)\big)$ for all $n\ge 0$.
        \item 
        If $\lambda(X)<\infty$, then $\mu^{n+1}_R(X)\mu(M)\le \big(\lambda(M)-\mu(M)\big) \mu^n_R(X) + \lambda\big(\Ext^{n+1}_R(M,X)\big)$ for all $n\ge 0$.
	\end{enumerate}
\end{lemma}

\begin{proof}
     (1) Consider a short exact sequence $0 \to \fm M \to M \to k^{\oplus \mu (M)} \to 0$.	It induces an exact sequence $\Tor_1^R(X,M) \to \Tor_1^R(X,k^{\oplus \mu(M)}) \to X \otimes_R \fm M $, which yields, by length consideration, that 
 	\begin{align}\label{beta-Tor}
 		\beta_{1}^R(X) \mu(M) \le  \lambda (X \otimes_R \fm M) +\lambda (\Tor_1^R(X,M)).
 	\end{align}
    Applying $(-) \otimes_R \fm M$ to the surjection $R^{\oplus \mu(X)} \to X$, one obtains another surjective homomorphism $R^{\oplus \mu(X)} \otimes_R \fm M \to X \otimes_R \fm M$, and hence $\lambda(X \otimes_R \fm M) \le \lambda( \fm M) \mu(X) =\big(\lambda(M)-\mu(M)\big) \beta_{0}^R(X)$.  Combining these inequalities with \eqref{beta-Tor}, one gets that
 \begin{align}\label{N Tor-N}
 	\beta_{1}^R(X) \mu(M) \le  \big(\lambda(M )-\mu(M)\big) \beta_{0}^R(X) +\lambda (\Tor_1^R(X,M)).
 \end{align}
 Note that $\beta_n^R(X)=\beta_0^R(\Omega_R^n(X))$, $\beta_{n+1}^R(X)=\beta_1^R(\Omega_R^n(X))$ and $\Tor^R_1(\Omega_R^n(X),M) \cong \Tor^R_{n+1}(X,M)$. So, considering $\Omega_R^n(X)$ in place of $X$, the inequality in (1) follows from \eqref{N Tor-N}.

 (2) Since $\lambda(M) < \infty$, one has that $\lambda(M^\vee)=\lambda(M) < \infty$, $\Ext_R^n(X,M)^\vee \cong \Tor^R_n(X,M^\vee)$ for all $n \ge 0$ (see \cite[3.6]{Hu07}), and $\type(M)=\mu_R^0(M)=\beta_0^R(M^\vee) = \mu(M^\vee)$. Using these results, the inequality in (1) for $X$ and $M^\vee$ yields the inequality in (2).
 
 (3) Since $X$ has finite length, so does $X^{\vee}$. Therefore, writing $X^{\vee}$ in place of $X$, the inequality in (1) gives that $\beta^R_{n+1}(X^{\vee})\mu(M)\le \big(\lambda(M)-\mu(M)\big) \beta_n^R(X^{\vee}) + \lambda\big(\Tor_{n+1}^R(X^{\vee},M)\big)$. Note that $\Tor_{n+1}^R(X^{\vee},M) \cong \Tor_{n+1}^R(M,X^{\vee}) \cong \Ext^{n+1}_R(M,X)^{\vee}$ and $\beta_n^R(X^{\vee}) = \mu^n_R(X)$. Hence the inequality in (3) follows.
\end{proof}

Using \Cref{lem:curv-x-y} and the inequalities in \Cref{lem:M-fin-len-ineq}, we prove the following proposition, which serve as the base case of \Cref{thm:main-cx-curv-CM-mod}.

\begin{proposition}\label{prop:ecx-ecurv-lambda-mu}
    Let $M$ and $X$ be nonzero $R$-modules such that $M$ has finite length.
    Then the following inequalities hold true.
    \begin{enumerate}[\rm (1)]
     \item 
     \begin{enumerate}[\rm (i)]
        \item $\curv_R(X) \le \sup\left\{\frac{\lambda(M)}{\mu(M)}-1, \tcurv_R(X,M)\right\}$.
        \item 
        If $\lambda(M)\le 2\mu(M) $, then $\cx_R(X) \le 1+\tcx_R(X,M)$.
        \item 
        If $\lambda(M)< 2\mu(M) $, then $\cx_R(X) \le \tcx_R(X,M)$ and $\curv_R(X) \le \tcurv_R(X,M)$.
     \end{enumerate}
    \item 
    \begin{enumerate}[\rm (i)]
        \item $\curv_R(X)\le \sup\left\{\frac{\lambda(M)}{\type(M)}-1, \curv_R(X,M)\right\}$.
        \item 
        If $\lambda(M)\le 2\type(M) $, then $\cx_R(X) \le 1+\cx_R(X,M)$.
        \item 
        If $\lambda(M)< 2\type(M) $, then $\cx_R(X) \le \cx_R(X,M)$ and $\curv_R(X) \le \curv_R(X,M)$.
    \end{enumerate}
    \item 
    Suppose that $X$ also has finite length. Then 
    \begin{enumerate}[\rm (i)]
        \item $\injcurv_R(X) \le \sup\left\{ \frac{\lambda(M)}{\mu(M)}-1, \curv_R(M,X)\right\}$.
        \item If $\lambda(M)\le 2\mu(M) $, then $\injcx_R(X)\le 1+\cx_R(M,X)$.
        \item If $\lambda(M)<2\mu(M) $, then $\injcx_R(X)\le \cx_R(M,X)$ and  $\injcurv_R(X)\le \curv_R(M,X)$.
    \end{enumerate}    
    \end{enumerate}
\end{proposition}

\begin{proof}
Note that $\lambda(M)\ge \max\{\mu(M),\type(M)\}$ (\Cref{thm:e-mu-type}).

(1) (i) In view of Lemma~\ref{lem:M-fin-len-ineq}.(1), $\beta_{n+1}^R(X) \le \big(\frac{\lambda(M)}{\mu(M)}-1\big) \beta_n^R(X) + \lambda\big(\Tor_{n+1}^R(X,M)\big)$ for all $n\ge 0$. Hence the first inequality of (i) follows from Corollary~\ref{cor:cx-curv-length} and Lemma~\ref{lem:curv-x-y}. Considering $X=k$ in the first inequality, one obtains the second inequality.

(ii)  Let $\lambda(M)\le 2\mu(M) $. Then $\frac{\lambda(M)}{\mu(M)}-1\le 1$. Hence, in view of Lemma~\ref{lem:M-fin-len-ineq}.(1), $\beta_{n+1}(X)-\beta_n(X)\le \lambda\big(\Tor_{n+1}^R(X,M)\big)$ for all $n\ge 0$. Therefore, by \cite[Prop.~2.2.(2)]{DV09}, one obtains that
\[
\cx_R(X) = \cx(\{\beta_n(X)\}) \le 1+\cx(\{\beta_{n+1}(X)-\beta_n(X)\})\le 1+\tcx_R(X,M).
\]
where the last inequality follows from Corollary~\ref{cor:cx-curv-length} and Remark~\ref{rmk:cx-curv-ineq}.(1).

(iii) Let $\lambda(M)< 2\mu(M) $. Then $\mu(M)>\lambda(M)-\mu(M)$. So, by \cite[Prop.~2.2.(4)]{DV09},
$$\cx\big(\big\{\mu(M)\beta_{n+1}(X)-(\lambda(M)-\mu(M))\beta_n(X)\big\}\big) = \cx(\{\beta_n(X)\}).$$
Hence, in view of Lemma~\ref{lem:M-fin-len-ineq}.(1), Corollary~\ref{cor:cx-curv-length} and Remark~\ref{rmk:cx-curv-ineq}, one concludes that
\[
    \cx_R(X) = \cx(\{\beta_n(X)\}) = \cx\big(\big\{\mu(M)\beta_{n+1}(X)-(\lambda(M)-\mu(M))\beta_n(X)\big\}\big) \le \tcx_R(X,M).
    \]
    To show $\curv_R(X) \le \tcurv_R(X,M)$, first notice that $\alpha:=\frac{\lambda(M)}{\mu(M)}-1 \in [0,1)$ as $\mu(M)\le \lambda(M)< 2\mu(M) $. If $\tcurv_R(X,M)\le \alpha$, then from (i), one gets that $\curv_R(X)\le \alpha <1$, and hence $\curv_R(X)=0\le \tcurv_R(X,M)$. If $\tcurv_R(X,M) > \alpha$, then $\curv_R(X)\le \tcurv_R(X,M)$ clearly follows from (i). 
    
    (2) and (3): The proofs are similar to that of (1) by using Lemma~\ref{lem:M-fin-len-ineq}.(2) and (3) respectively in place of Lemma~\ref{lem:M-fin-len-ineq}.(1).
\end{proof}

Now we are in a position to prove our main result of this section.

\begin{theorem}\label{thm:main-cx-curv-CM-mod}
Let $M$ and $X$ be nonzero $R$-modules such that $M$ is CM. Then, the following hold.
\begin{enumerate}[\rm(1)]
    \item 
    Suppose there exists a positive integer $h$ satisfying $\fm^h \Tor_n^R(X,M)=0$ for all $n \gg 0$. Then
    \begin{enumerate}[\rm (i)]
        \item $\curv_R(X) \le \sup\left\{\frac{e(M)}{\mu(M)}-1, \tcurv_R(X,M)\right\}.$
        \item If $e(M)\le 2 \mu(M) $, then $\cx_R(X) \le 1+\tcx_R(X,M)$. 
        \item 
        If $e(M)< 2\mu(M) $, then $\cx_R(X) \le \tcx_R(X,M)$ and $\curv_R(X) \le \tcurv_R(X,M)$.
     \end{enumerate}
    \item 
    Suppose there exists a positive integer $h$ satisfying $\fm^h \Ext^n_R(X,M)=0$ for all $n \gg 0$. Then
    \begin{enumerate}[\rm (i)]
        \item $\curv_R(X)\le \sup\left\{\frac{e(M)}{\type(M)}-1, \curv_R(X,M)\right\}.$
        \item 
        If $e(M)\le 2\type(M) $, then $\cx_R(X) \le 1+\cx_R(X,M)$.
        \item 
        If $e(M)< 2\type(M) $, then $\cx_R(X) \le \cx_R(X,M)$ and $\curv_R(X) \le \curv_R(X,M)$.
    \end{enumerate}
    \item
    Suppose that $X$ is also CM, and there exists an integer $h>0$ satisfying $\fm^h \Ext^n_R(M,X)=0$ for all $n \gg 0$. Then
    \begin{enumerate}[\rm (i)]
        \item $\injcurv_R(X) \le \sup\left\{ \frac{e(M)}{\mu(M)}-1, \curv_R(M,X)\right\}$.
        \item If $e(M)\le 2\mu(M) $, then $\injcx_R(X)\le 1+\cx_R(M,X)$.
        \item If $e(M)<2\mu(M) $, then $\injcx_R(X)\le \cx_R(M,X)$ and  $\injcurv_R(X)\le \curv_R(M,X)$.
    \end{enumerate}
\end{enumerate}
\end{theorem}

\begin{proof}
Due to \Cref{rmk:flat-ext-CM}, we may pass to the faithfully flat extension $R[X]_{\fm[X]}$, and assume that $R$ has infinite residue field. By the same argument as in \cite[Rmk.~2.4]{DG23}, $e(M)=\lambda(M/\mathbf x M)$ for some system of parameters ${\bf x}$ of $M$ (and hence for some $M$-regular sequence ${\bf x}$). Clearly, $M/\mathbf x M$ has finite length, $\mu(M)=\mu(M/\mathbf x M)$ and $\type(M)=\type(M/\mathbf x M)$ (cf.~\cite[1.2.4]{BH98}).

(1) By Lemma~\ref{lem:cx-curv-M-M/xM-Tor}, $ \tcx_R(X,M/\mathbf x M)\le \tcx_R(X,M) $ and $ \tcurv_R(X,M/{\mathbf x}M) \le \tcurv_R(X,M) $. Hence the desired inequalities follow by using Proposition~\ref{prop:ecx-ecurv-lambda-mu}.(1) for $X$ and $M/{\bf x}M$.

(2) By Lemma~\ref{lem:cx-curv-M-M/xM-Ext}.(3), $\cx_R(X,M/{\mathbf x}M)\le \cx_R(X,M)$ and $\curv_R(X,M/{\mathbf x}M)\le \curv_R(X,M)$. Hence  the desired inequalities follow by using Proposition~\ref{prop:ecx-ecurv-lambda-mu}.(2) for $X$ and $M/{\bf x}M$.  

(3) Since $X$ is CM, there is an $X$-regular sequence $\mathbf y$ such that $X/\mathbf y X$ has finite length. Then, a repeated use of \Cref{lem:cx-curv-M-M/xM}.(2) yields that $\injcx_R(X)=\injcx_R(X/\mathbf y X)$ and $\injcurv_R(X)=\injcurv_R(X/\mathbf y X)$. Moreover, by \Cref{lem:cx-curv-M-M/xM-Ext}.(3), $\cx_R(M/{\mathbf x}M,X/\mathbf y X)\le \cx_R(M,X)$ and $\curv_R(M/{\mathbf x}M,X/\mathbf y X)\le \curv_R(M,X)$. Hence, using Proposition~\ref{prop:ecx-ecurv-lambda-mu}.(3) for $M/{\bf x}M$ and $X/{\bf y}X$, the claim follows.
\end{proof}
When $R$ is CM, the CM assumption on $X$ can be removed from Theorem~\ref{thm:main-cx-curv-CM-mod}.(3) 

\begin{corollary}\label{cor:mcm-aprrox}
Let $R$ be CM, and $M$ be a nonzero CM $R$-module. Let $X$ be an $R$-module such that there exists an integer $h>0$ satisfying $\fm^h \Ext^n_R(M,X)=0$ for all $n\gg 0$. Then, the following hold.
\begin{enumerate}[\rm (1)]
    \item $\injcurv_R(X) \le \sup\left\{ \frac{e(M)}{\mu(M)}-1, \curv_R(M,X)\right\}$.
    \item If $e(M)\le 2\mu(M) $, then $\injcx_R(X)\le 1+\cx_R(M,X)$.
    \item If $e(M)<2\mu(M) $, then $\injcx_R(X)\le \cx_R(M,X)$ and  $\injcurv_R(X)\le \curv_R(M,X)$.
\end{enumerate}    
\end{corollary}

\begin{proof}
We may pass to completion, and assume that $R$ is complete. Hence $R$ admits a canonical module. Thus, by \cite[Thm.~11.17]{LW}, there exists an MCM $R$-module $X'$ and an $R$-module $Y$ of finite injective dimension which fit into an exact sequence $0\to Y\to X'\to X\to 0$. It follows that
\begin{equation}\label{iso-X-X'}
    \Ext^n_R(N,X')\cong \Ext^n_R(N,X) \mbox{ for every $R$-module $N$ and for all $n>\dim R$.}
\end{equation}
Considering $N=k$ in \eqref{iso-X-X'}, one gets that $\injcx_R(X')=\injcx_R(X)$ and $\injcurv_R(X')=\injcurv_R(X)$. Moreover, the isomorphisms in \eqref{iso-X-X'} for $N=M$ yields that $\cx_R(M,X')=\cx_R(M,X)$, $\curv_R(M,X')=\curv_R(M,X)$ and $\fm^h \Ext^n_R(M,X')=0$ for all $n\gg 0$. Thus, we may replace $X$ by $X'$, and assume that $X$ is MCM. Now we are done by Theorem~\ref{thm:main-cx-curv-CM-mod}.(3).   
\end{proof}

It is classically known that complexity and curvature of any module is bounded above by that of the residue field $k$, see \Cref{prop:cx-curv-facts}.(1). As a consequence of \Cref{thm:main-cx-curv-CM-mod}, we can bound the curvature of $k$ in terms of that of any nonzero CM module. Note that curvature of a module is always finite (\cite[4.2.3.(5)]{Avr98}), which is not the case in general for complexity.   

\begin{corollary}\label{cor:cx-curv-k-CM-mod}
Let $M$ be a nonzero CM $R$-module. Then, the following hold.
\begin{enumerate}[\rm(1)]
    \item $\curv_R(k) \le \sup\left\{\frac{e(M)}{\mu(M)}-1, \curv_R(M)\right\}.$
    \item $\curv_R(k) \le \sup\left\{\frac{e(M)}{\type(M)}-1, \injcurv_R(M)\right\}.$
    \item If $e(M)\le 2\mu(M) $, then $\cx_R(k)\le 1+\cx_R(M)$, i.e., $\cx_R(k)=\cx_R(M)$ or $1+\cx_R(M)$.
    \item If $e(M)<2\mu(M) $, then $\cx_R(k) = \cx_R(M)$ and  $\curv_R(k) = \curv_R(M)$.
    \item If $e(M)\le 2\type(M) $, then $\cx_R(k)\le 1+\injcx_R(M)$, i.e., $\cx_R(k)=\injcx_R(M)$ or $1+\injcx_R(M)$.
    \item If $e(M)<2\type(M) $, then $\cx_R(k)=\injcx_R(M)$ and  $\curv_R(k)=\injcurv_R(M)$.
\end{enumerate}
\end{corollary}

\begin{proof}
    Note that
    $\tcurv_R(k,M) = \curv_R(M)$ and $\tcx_R(k,M) = \cx_R(M)$. Moreover, it is always true that $\cx_R(M)\le \cx_R(k)$ and $\curv_R(M)\le \curv_R(k)$ (cf.~\Cref{prop:cx-curv-facts}.(1)). Therefore (1), (3) and (4) can be obtained from Theorem~\ref{thm:main-cx-curv-CM-mod}.(1) by considering $X=k$. For (2), (5) and (6), notice that $\curv_R(k,M) = \injcurv_R(M)$ and $\cx_R(k,M) = \injcx_R(M)$. Also note that $\injcx_R(M)\le \cx_R(k)$ and $\injcurv_R(M)\le \curv_R(k)$. Hence (2), (5) and (6) follow from Theorem~\ref{thm:main-cx-curv-CM-mod}.(2) by considering $X=k$.
\end{proof}

\begin{remark}\label{rmk:3-examples}
    Examples~\ref{exam-2-cx-curv-inequality} and \ref{exam-1-cx-curv-inequality} respectively ensure that the inequalities in \Cref{cor:cx-curv-k-CM-mod}.(1) and (2) (hence in \Cref{thm:main-cx-curv-CM-mod}) are sharp. Moreover, these examples show that the conditions `$e(M) \le 2 \mu(M)$' and `$e(M) \le 2 \type(M)$' cannot be omitted from \Cref{cor:cx-curv-k-CM-mod}.(3) and (5) respectively. The conditions `$e(M) < 2 \mu(M)$' and `$e(M) < 2 \type(M)$' in \Cref{cor:cx-curv-k-CM-mod}.(4) and (6) (hence in \Cref{thm:main-cx-curv-CM-mod}) cannot be replaced by `$e(M) \le 2 \mu(M)$' and `$e(M) \le 2 \type(M)$' respectively as shown in \Cref{exam:strict-ineq-cant-be-eq}.
\end{remark}

\begin{example}\label{exam-2-cx-curv-inequality}
    Let $R = k[x_1,\ldots,x_b,y]/((x_1,\ldots,x_b)^2,y^2)$ over a field $k$, where $b\ge 2$. Then $R$ is an Artinian local ring, which is not Gorenstein as $\Soc(R) = (x_1y,\ldots,x_by) \cong k^{\oplus b}$. Here $\fm = (x_1,\ldots,x_b,y)$. Set $M:=(y)$. Then $\fm^2 M = 0$, i.e., $M$ has minimal multiplicity. Note that $e(M) = \lambda(M) = 1+b > 2 = 2 \mu(M)$. Since $Ry=\ann_R(y)$, one has that $\beta_n^R(M)=1$ for all $n\ge 0$, $\cx_R(M) = 1 = \curv_R(M)$. Next we compute $\cx_R(k)$ and $\curv_R(k)$. Let $\mathbb{F}_U$ and $\mathbb{F}_V$ be minimal $R$-free resolutions of $U:=R/(x_1,\ldots,x_b)$ and $V:=R/(y)$ respectively. Notice that $\beta_n^R(V) = \beta_n^R(M) = 1$. Also $\Omega^1_R(U) = (x_1,\ldots,x_b) = (x_1)\oplus\cdots\oplus (x_b)\cong U^{\oplus b}$, which yields that $\beta_n^R(U) = b^n$ for all $n\ge 0$. Since $k \cong U\otimes_R V$ and $\Tor_i^R\big(U,V\big)=0$ for all $i\ge 1$, the tensor product of $\mathbb{F}_U$ and $\mathbb{F}_V$ yields a minimal $R$-free resolution $\mathbb{F}_k$ of $k$, i.e., $\mathbb{F}_k \cong \mathbb{F}_U \otimes_R \mathbb{F}_V $. It follows that $\beta_n^R(k)=\sum_{i=0}^n\beta_i^R(U)\beta_{n-i}^R(V) = \sum_{i=0}^n b^i=b^{n+1}-1$ for all $n\ge 0$. This implies that $\cx_R(k) = \infty$ and $ \curv_R(k) = b$. Thus $\cx_R(k) > 1+\cx_R(M)$ and
    \begin{center}
        $ \sup\left\{\frac{e(M)}{\mu(M)}-1, \curv_R(M)\right\} = \frac{e(M)}{\mu(M)}-1 = b = \curv_R(k) $.
    \end{center}
\end{example}

\begin{example}\label{exam-1-cx-curv-inequality}
    Let $(R,\fm,k)$ be an Artinian local ring such that $\fm^2=0$ and $\mu(\fm)=b\ge 2$ (e.g., $R$ can be $k[x_1,\ldots,x_b]/\langle x_ix_j : 1\le i\le j\le b\rangle$, where $b\ge 2$). Then $R$ is not Gorenstein as $\type(R) = \mu(\fm) =b \neq 1$. Set $E:=E_R(k)$, the injective hull of $k$. Clearly, $\fm^2 E = 0$, i.e., $E$ has minimal multiplicity. By Matlis duality, $e(E)=\lambda(E) = \lambda(R) = 1 + \mu(\fm)=1+b$, $\mu(E) = \type(R) = \mu(\fm)=b$ and $\type(E)=1$. Therefore, since $b\ge 2$, one has that $e(E) = 1+b > 2 = 2\type(E)$ and $e(E) = 1+b < 2b = 2 \mu(E)$. Note that $\injcx_R(E) = 0 = \injcurv_R(E)$ as $E$ is injective. Also $\Omega_R^1(k) = \fm \cong k^{\oplus b}$ as $\fm^2=0$. Hence, inductively, $\Omega_R^n(k) \cong k^{\oplus b^n}$ for all $n\ge 1$. So $\beta_n^R(k) = b^n$ for all $n\ge 0$. This yields that $\cx_R(k) = \infty$ and $ \curv_R(k) = b$. Thus $\cx_R(k) > 1+\injcx_R(E)$. Moreover, $\frac{e(E)}{\type(E)}-1 = b = \frac{\mu(E)}{e(E)-\mu(E)}$, and hence
    \begin{center}
        $\sup\left\{\frac{e(E)}{\type(E)}-1, \injcurv_R(E)\right\} = b=\curv_R(k) = \sup\left\{\frac{\mu(E)}{e(E)-\mu(E)}, \injcurv_R(E)\right\}.$
    \end{center}
    Note that $e(R)=\lambda(R)=1+b<2b=2\type(R)$, $\cx_R(k) > 1 + \cx_R(R)$ and
    \begin{center}
        $\curv_R(k) =b=\sup\left\{\frac{\type(R)}{e(R)-\type(R)}, \curv_R(R)\right\}.$
    \end{center}
\end{example}


\begin{example}\label{exam:strict-ineq-cant-be-eq}
    Let $R$ be an abstract hypersurface of multiplicity $2$ (e.g., $R=k[x]/(x^2)$). Then $M:=R$ is a CM $R$-module of minimal multiplicity such that $e(M)=2=2\mu(M)=2\type(M)$. Note that $\cx_R(k)=1\neq 0 = \cx_R(M)= \injcx_R(M)$ and $\curv_R(k)=1\neq 0 = \curv_R(M)= \injcurv_R(M)$.
\end{example}

The next corollary shows that for a large class of modules, the condition on annihilation of Tor or Ext  modules by some power of the maximal ideal as in \Cref{thm:main-cx-curv-CM-mod} and \Cref{cor:mcm-aprrox} can be dropped.

\begin{corollary}\label{cor:loc-fin-pd}
Let $M$ and $X$ be nonzero $R$-modules such that $M$ is CM.
\begin{enumerate}[\rm(1)]
    \item 
    Suppose either $M$ or $X$ is locally of finite projective dimension on $\Spec(R) \smallsetminus \{ \fm \}$. Then
    \begin{enumerate}[\rm (i)]
        \item $\curv_R(X) \le \sup\left\{\frac{e(M)}{\mu(M)}-1, \tcurv_R(X,M)\right\}.$
        \item If $e(M)\le 2 \mu(M) $, then $\cx_R(X) \le 1+\tcx_R(X,M)$. 
        \item 
        If $e(M)< 2\mu(M) $, then $\cx_R(X) \le \tcx_R(X,M)$ and $\curv_R(X) \le \tcurv_R(X,M)$.
     \end{enumerate}
    \item 
    Suppose $X$ is locally of finite projective dimension on $\Spec(R) \smallsetminus \{ \fm \}$. Then
    \begin{enumerate}[\rm (i)]
        \item $\curv_R(X)\le \sup\left\{\frac{e(M)}{\type(M)}-1, \curv_R(X,M)\right\}.$
        \item 
        If $e(M)\le 2\type(M) $, then $\cx_R(X) \le 1+\cx_R(X,M)$.
        \item 
        If $e(M)< 2\type(M) $, then $\cx_R(X) \le \cx_R(X,M)$ and $\curv_R(X) \le \curv_R(X,M)$.
    \end{enumerate}
    \item 
    Suppose either $X$ is CM, or $R$ is CM. Also assume that $M$ is locally of finite projective dimension on $\Spec(R) \smallsetminus \{ \fm \}$. Then
    \begin{enumerate}[\rm (i)]
        \item $\injcurv_R(X) \le \sup\left\{ \frac{e(M)}{\mu(M)}-1, \curv_R(M,X)\right\}$.
        \item If $e(M)\le 2\mu(M) $, then $\injcx_R(X)\le 1+\cx_R(M,X)$.
        \item If $e(M)<2\mu(M) $, then $\injcx_R(X)\le \cx_R(M,X)$ and  $\injcurv_R(X)\le \curv_R(M,X)$.
    \end{enumerate}
\end{enumerate}    
\end{corollary}

\begin{proof}
(1) In view of Lemma~\ref{lem:Tor-m-zero}.(2), there exists an integer $h>0$ such that $\fm^h \Tor_n^R(X,M)=0$ for all $n>\dim(R)$. Hence we are done by \Cref{thm:main-cx-curv-CM-mod}.(1).

(2) In view of Lemma~\ref{lem:Tor-m-zero}.(1), there exists an integer $g>0$ such that $\fm^g \Ext^n_R(X,M)=0$ for all $n>\dim(R)$. Hence the desired inequalities follow from Theorem~\ref{thm:main-cx-curv-CM-mod}.(2).

(3) In view of Lemma~\ref{lem:Tor-m-zero}.(1), there exists an integer $f>0$ such that $\fm^f \Ext^n_R(M,X)=0$ for all $n>\dim(R)$. Therefore, when $X$ is CM, we are done by Theorem~\ref{thm:main-cx-curv-CM-mod}.(3), and when $R$ is CM, the desired inequalities follow from Corollary~\ref{cor:mcm-aprrox}.
\end{proof}

Our next result shows that every nonzero CM module $M$ with $ e(M) \le 2 \mu(M) $ or $e(M)\le 2\type(M) $ can be used to characterize complete intersection rings. Note that \Cref{cor:char-CI-rings}.(1) considerably strengthens \cite[Cor.~3.9]{DG23}, which was proved for Ulrich modules (\Cref{defn:Ulrich-mod}) over a CM local ring.

\begin{corollary}\label{cor:char-CI-rings}
    Let $M$ be a nonzero CM $R$-module.
    \begin{enumerate}[\rm (1)]
        \item Let $ e(M) \le 2 \mu(M) $. Then the following are equivalent:\\
        $\rm (a)$~$R$ is complete intersection; $\rm (b)$~$\cx_R(M)<\infty$; and $\rm (c)$~$\curv_R(M)\le 1$.
        \item Let $e(M)\le 2\type(M) $. Then the following are equivalent:\\
        $\rm (a)$~$R$ is complete intersection; $\rm (b)$~$\injcx_R(M)<\infty$; and $\rm (c)$~$\injcurv_R(M)\le 1$.
    \end{enumerate}
\end{corollary}

\begin{proof}
    The forward implications in (1) and (2) are clear due to \Cref{prop:cx-curv-facts}.(3) and \Cref{rmk:cx-curv}.\eqref{finite-cx-implies-curv}. It remains to prove (c) $\Rightarrow$ (a) in (1), and (c) $\Rightarrow$ (a) in (2). These implications follow from \Cref{cor:cx-curv-k-CM-mod}.(1) and (2) respectively, and \Cref{prop:cx-curv-facts}.(2). 
\end{proof}


When $R$ is CM, taking $M=R$ in Theorem~\ref{thm:main-cx-curv-CM-mod}.(1).(i) and Corollary~\ref{cor:mcm-aprrox}.(1), we obtain the following result which contains \cite[Prop.~4.2.6]{Avr98}.  

\begin{corollary}
If $R$ is CM, then $\max\{\curv(X),\injcurv(X)\} \le e(R)-1$ for every $R$-module $X$.
\end{corollary}



\section{On complexity and curvature of (pair of) CM modules of minimal multiplicity} \label{sec:cx-curv-cm-mod-min-mult}

In this section, we complement  the results of \Cref{sec:cx-curv-cm-mod} by proving that CM modules of minimal multiplicity are not far from having maximal (projective or injective) complexity and curvature, provided $ e(M) \ge 2 \type(M)$ or $e(M)\ge 2 \mu(M)$ holds respectively. As in the previous section, our main results in this section, \Cref{thm:cx-curv-min-mult} and \Cref{Cor:cx-curv-min-CM}, are also stated more generally, in terms of complexity and curvature of pair of modules. 

The second part of the following lemma should be compared with Lemma~\ref{lem:M-fin-len-ineq}.(3).

\begin{lemma}\label{lem:Bass-k-M-inequality}
    Let $M$ be a nonzero $R$-module such that $\fm^2 M=0$. Let $X$ be an $R$-module. Then
    \begin{enumerate}[\rm (1)]
        \item $\big(\lambda(M)-\mu(M)\big)\beta_{n+1}^R(X)-\mu(M)\beta_{n}^R(X)\le \lambda\big(\Ext^{n+1}_R(X,M)\big)$ for all $n\ge 0$.
        \item $\mu(M)\mu_R^{n+1}(X)-\big(\lambda(M)-\mu(M)\big)\mu_R^n(X)\le \lambda\big(\Ext^{n+1}_R(M,X)\big)$ for all $n\ge 0$.
        \item $\big(\lambda(M)-\type(M)\big)\beta_{n+1}^R(X)-\type(M)\beta_{n}^R(X)\le \lambda\big(\Tor_{n+1}^R(X,M)\big)$ for all $n\ge 0$.
        \item If $\lambda(X)<\infty$, then $\big(\lambda(M)-\type(M)\big)\mu^{n+1}_R(X)-\type(M)\mu^{n}_R(X)\le \lambda\big(\Ext^{n+1}_R(M,X)\big)$ for all $n\ge 0$.
    \end{enumerate}
\end{lemma} 

\begin{proof}
    Set $u :=\mu(\fm M)$ and $v:=\mu(M)$. Since $\fm^2 M=0$, it follows that $u=\lambda(\fm M)=\lambda(M)-\mu(M)$. Moreover, there is a short exact sequence $0 \to k^{\oplus u} \to M \to k^{\oplus v} \to 0$, which yields the exact sequences $\Ext_R^n(X,k^{\oplus v}) \rightarrow \Ext_R^{n+1}(X,k^{\oplus u}) \rightarrow \Ext_R^{n+1}(X,M)$ and $\Ext_R^n(k^{\oplus u},X) \rightarrow \Ext_R^{n+1}(k^{\oplus v},X) \rightarrow \Ext_R^{n+1}(M,X)$ for all $n\ge 0$. Hence, by length consideration, $u \beta_{n+1}^R(X) \le v \beta_n^R(X) + \lambda(\Ext_R^{n+1}(X,M))$ and $v\mu_R^{n+1}(X)\le u\mu_R^n(X)+\lambda(\Ext^{n+1}_R(M,X))$ for all $n\ge 0$. Thus the desired inequalities (1) and (2) follow. The inequalities in (3) follow from (1) by remembering $\Tor^R_i(X,M)^{\vee}\cong \Ext_R^{i}(X,M^{\vee})$ for all $i\ge 0$, $\fm^2M=0$ if and only if $\fm^2(M^{\vee})=0$, $\lambda(M)=\lambda(M^{\vee})$ and $\type(M)=\mu(M^{\vee})$. For (4), since $\lambda(X)<\infty$, one has that $\mu^i_R(X)=\beta^R_i(X^{\vee})$ and $\Ext^i_R(M,X)^{\vee}\cong \Tor^R_i(M,X^{\vee})\cong \Tor^R_i(X^{\vee},M)$ for all $i\ge 0$. Hence we are done by applying (3) on $M$ and $X^{\vee}$.
\end{proof}

The following proposition is the base case of \Cref{thm:cx-curv-min-mult}.

\begin{proposition}\label{prop:main-base-case}
Let $M$ and $X$ be nonzero $R$-modules such that $\fm^2 M=0$. Then, the following hold.
\begin{enumerate}[\rm(1)] 
    \item 
    \begin{enumerate}[\rm(i)]
        \item $\curv_R(X)\le \sup\left\{\frac{\mu(M)}{\lambda(M)-\mu(M)}, \curv_R(X,M)\right\}$ provided $\fm M\neq 0$.
        \item If $\lambda(M)\ge 2\mu(M)$, then $\cx_R(X)\le 1+\cx_R(X,M)$.
        \item If $\lambda(M)>2\mu(M)$, then $\cx_R(X)\le \cx_R(X,M)$ and $\curv_R(X)\le \curv_R(X,M)$. 
    \end{enumerate}
\item 
\begin{enumerate}[\rm (i)]
        \item $\injcurv_R(X) \le \sup\left\{ \frac{\lambda(M)}{\mu(M)}-1, \curv_R(M,X)\right\}$.
        \item If $\lambda(M)\le 2\mu(M) $, then $\injcx_R(X)\le 1+\cx_R(M,X)$.
        \item If $\lambda(M)<2\mu(M) $, then $\injcx_R(X)\le \cx_R(M,X)$ and  $\injcurv_R(X)\le \curv_R(M,X)$.
\end{enumerate}
    \item 
    \begin{enumerate}[\rm(i)]
        \item $\curv_R(X)\le \sup\left\{\frac{\type(M)}{\lambda(M)-\type(M)}, \tcurv_R(X,M)\right\}$ provided $\fm M\neq 0$.
        \item If $\lambda(M)\ge 2\type(M)$, then $\cx(X)_R\le 1+\tcx_R(X,M)$.
        \item If $\lambda(M)>2\type(M)$, then $\cx_R(X)\le \tcx_R(X,M)$ and $\curv_R(X)\le \tcurv_R(X,M)$. 
    \end{enumerate}    
    \item Let $\lambda (X) < \infty$. Then
    \begin{enumerate}[\rm(i)]
        \item $\injcurv_R(X)\le \sup\left\{\frac{\type(M)}{\lambda(M)-\type(M)}, \curv_R(M,X)\right\}$ provided $\fm M\neq 0$.
        \item If $\lambda(M)\ge 2\type(M)$, then $\injcx_R(X)\le 1+\cx_R(M,X)$.
        \item If $\lambda(M)>2\type(M)$, then $\injcx_R(X)\le \cx_R(M,X)$ and $\injcurv_R(X)\le \curv_R(M,X)$. 
    \end{enumerate}
\end{enumerate}    
\end{proposition}

\begin{proof}
Note that $\lambda(M)\ge \lambda(M/\fm M) = \mu(M)$. Moreover, $\fm M\neq 0$ \iff $\lambda(M)>\mu(M)$, which is equivalent to that $\lambda(M)>\type(M)$, see \Cref{thm:e-mu-type}.
    
(1) (i) Let $\fm M\neq 0$. Then, by Lemma~\ref{lem:Bass-k-M-inequality}.(1), $\beta_{n+1}^R(X)\le \frac{\mu(M)}{\lambda(M)-\mu(M)}\beta_{n}^R(X) + \lambda\big(\Ext^{n+1}_R(X,M)\big)$ for all $n\ge 0$. Hence the desired inequality follows from Corollary~\ref{cor:cx-curv-length} and Lemma~\ref{lem:curv-x-y}. 

(ii) Let $\lambda(M)\ge 2\mu(M) $. Then $(\lambda(M)-\mu(M))\ge \mu(M)$. So, in view of Lemma~\ref{lem:Bass-k-M-inequality}.(1), $\beta_{n+1}^R(X)-\beta_n^R(X) \le \lambda(\Ext_R^{n+1}(X,M))$ for all $n\ge 0$. Hence \cite[Prop.~2.2.(2)]{DV09} yields that
\[
\cx_R(X) = \cx(\{\beta_n^R(X)\}) \le 1+\cx(\{\beta_{n+1}^R(X)-\beta_n^R(X)\})\le 1+\cx_R(X,M),
\]
where the last inequality follows from Corollary~\ref{cor:cx-curv-length} and Remark~\ref{rmk:cx-curv-ineq}.(1).

(iii) Let $\lambda(M)> 2\mu(M) $. Then $(\lambda(M)-\mu(M))>\mu(M)$. So, by \cite[Prop.~2.2.(4)]{DV09},
$$\cx\big(\big\{(\lambda(M)-\mu(M))\beta_{n+1}^R(X)-\mu(M)\beta_n^R(X)\big\}\big) = \cx(\{\beta_n^R(X)\}).$$
Hence, in view of Lemma~\ref{lem:Bass-k-M-inequality}.(1), Corollary~\ref{cor:cx-curv-length} and Remark~\ref{rmk:cx-curv-ineq}.(1), one concludes that
\[
  \cx_R(X) = \cx(\{\beta_n^R(X)\}) = \cx\big(\big\{(\lambda(M)-\mu(M))\beta_{n+1}^R(X)-\mu(M)\beta_n^R(X)\big\}\big) \le \cx_R(X,M).
\]
To show $\curv_R(X)\le \curv_R(X,M)$, first notice that $\alpha:=\frac{\mu(M)}{\lambda(M)-\mu(M)} \in (0,1)$ as $\lambda(M)>2\mu(M)$. If $\curv_R(X,M)\le \alpha$, then from (i), one gets that $\curv_R(X)\le \alpha <1$, and hence $\curv_R(X)=0\le \curv_R(X,M)$. If $\curv_R(X,M)> \alpha$, then  $\curv_R(X)\le \curv_R(X,M)$ clearly follows from (i).
%
%
%

    (2), (3) and (4): The proofs are similar to that of (1) by using Lemma~\ref{lem:Bass-k-M-inequality}.(2), (3) and (4) respectively in place of Lemma~\ref{lem:Bass-k-M-inequality}.(1).
\end{proof}

Now we prove the main result of this section.

\begin{theorem}\label{thm:cx-curv-min-mult}
Let $M$ and $X$ be nonzero $R$-modules such that $M$ has minimal multiplicity.
\begin{enumerate}[\rm(1)] 
    \item Suppose there exists an integer $h>0$ such that $\fm^h \Ext^n_R(X,M)=0$ for all $n\gg 0$. Then
    \begin{enumerate}[\rm(i)]
        \item $\curv_R(X)\le \sup\left\{\frac{\mu(M)}{e(M)-\mu(M)}, \curv_R(X,M)\right\}$ provided $M$ is not Ulrich.
        \item If $e(M)\ge 2\mu(M)$, then $\cx_R(X)\le 1+\cx_R(X,M)$.
        \item If $e(M)>2\mu(M)$, then $\cx_R(X)\le \cx_R(X,M)$ and $\curv_R(X)\le \curv_R(X,M)$. 
    \end{enumerate}
    \item Suppose there exists an integer $h>0$ such that $\fm^h \Ext^n_R(M,X)=0$ for all $n\gg 0$. Then 
    \begin{enumerate}[\rm (i)]
        \item $\injcurv_R(X) \le \sup\left\{ \frac{e(M)}{\mu(M)}-1, \curv_R(M,X)\right\}$.
        \item If $e(M)\le 2\mu(M) $, then $\injcx_R(X)\le 1+\cx_R(M,X)$.
        \item If $e(M)<2\mu(M) $, then $\injcx_R(X)\le \cx_R(M,X)$ and  $\injcurv_R(X)\le \curv_R(M,X)$.
    \end{enumerate}
    \item Suppose there exists an integer $h>0$ such that $\fm^h \Tor_n^R(X,M)=0$ for all $n\gg 0$. Then
    \begin{enumerate}[\rm(i)]
        \item $\curv_R(X)\le \sup\left\{\frac{\type(M)}{e(M)-\type(M)}, \tcurv_R(X,M)\right\}$ provided $M$ is not Ulrich.
        \item If $e(M)\ge 2\type(M)$, then $\cx_R(X)\le 1+\tcx_R(X,M)$.
        \item If $e(M)>2\type(M)$, then $\cx_R(X)\le \tcx_R(X,M)$ and $\curv_R(X)\le \tcurv_R(X,M)$. 
    \end{enumerate}
    \item Suppose that $X$ is CM, and there exists an integer $h>0$ such that $\fm^h \Ext^n_R(M,X)=0$ for all $n\gg 0$. Then
    \begin{enumerate}[\rm(i)]
        \item $\injcurv_R(X)\le \sup\left\{\frac{\type(M)}{e(M)-\type(M)}, \curv_R(M,X)\right\}$ provided $M$ is not Ulrich.
        \item If $e(M)\ge 2\type(M)$, then $\injcx_R(X)\le 1+\cx_R(M,X)$.
        \item If $e(M)>2\type(M)$, then $\injcx_R(X)\le \cx_R(M,X)$ and $\injcurv_R(X)\le \curv_R(M,X)$. 
    \end{enumerate}  
\end{enumerate}    
\end{theorem}

\begin{proof}
Due to Remarks~\ref{rmk:flat-ext-CM} and \ref{rmk:flat-ext-Ulrich}, we may pass to the faithfully flat extension $R[X]_{\fm[X]}$, and assume that $R$ has infinite residue field. Therefore, by Lemma~\ref{lem:mult-inequality}, there exists a system of parameters ${\bf x}$ of $M$ such that $e(M)=\lambda(M/\mathbf x M)$ and $\fm^2  (M/\mathbf x M)=0$. Since $M$ is CM, ${\bf x}$ is also a (maximal) $M$-regular sequence. Then, $M/\mathbf x M$ has finite length, $\mu(M)=\mu(M/\mathbf x M)$ and $\type(M)=\type(M/\mathbf x M)$ (cf.~\cite[1.2.4]{BH98}).  Moreover, if $M$ is not Ulrich, then $\fm(M/\mathbf xM)\neq 0$ by \cite[Prop.~2.2.(2)]{GTT15}. The inequalities given in (1), (2), and (3) now follow from Lemmas~\ref{lem:cx-curv-M-M/xM-Tor} and \ref{lem:cx-curv-M-M/xM-Ext} by using Proposition~\ref{prop:main-base-case}.(1), (2), and (3) respectively for $M/\mathbf xM$ and $X$. For (4), consider a maximal $X$-regular sequence $\mathbf{y}$. Then $\lambda(X/\mathbf{y}X)< \infty$. Hence (4) follows from Lemmas~\ref{lem:cx-curv-M-M/xM}  and \ref{lem:cx-curv-M-M/xM-Ext} by using Proposition~\ref{prop:main-base-case}.(4) for $M/{\mathbf x}M$ and $X/\mathbf{y}X$.
\end{proof}

\begin{remark}
    Comparing Theorems~\ref{thm:main-cx-curv-CM-mod}.(3) and \ref{thm:cx-curv-min-mult}.(2), note that under the assumption of minimal multiplicity of $M$ (as in \ref{thm:cx-curv-min-mult}.(2)), we can remove the CM condition on $X$ from \Cref{thm:main-cx-curv-CM-mod}.(3).
\end{remark}
When $R$ is CM, the CM assumption on $X$ can be removed from \Cref{thm:cx-curv-min-mult}.(4)

\begin{corollary}\label{Cor:cx-curv-min-CM}
Assume that $R$ is CM. Let $M$ and $X$ be nonzero $R$-modules such that $M$ has minimal multiplicity. Suppose there exists an integer $h>0$ such that $\fm^h \Ext^n_R(M,X)=0$ for all $n\gg 0$. Then
    \begin{enumerate}[\rm(i)]
        \item $\injcurv_R(X)\le \sup\left\{\frac{\type(M)}{e(M)-\type(M)}, \curv_R(M,X)\right\}$ provided $M$ is not Ulrich. 
        \item If $e(M)\ge 2\type(M)$, then $\injcx_R(X)\le 1+\cx_R(M,X)$.
        \item If $e(M)>2\type(M)$, then $\injcx_R(X)\le \cx_R(M,X)$ and $\injcurv_R(X)\le \curv_R(M,X)$. 
    \end{enumerate}    
\end{corollary}

\begin{proof}
The proof is similar to that of \Cref{cor:mcm-aprrox} by using \Cref{thm:cx-curv-min-mult}.(4) in place of \ref{thm:main-cx-curv-CM-mod}.(3).
\end{proof}

Taking $X=k$ in Theorem~\ref{thm:cx-curv-min-mult}.(1) and (3), one immediately obtains the following.

\begin{corollary}\label{cor:main-cx-curv-CM-min-mult}
    Let $M$ be a nonzero $R$-module of minimal multiplicity. Then, the following hold.
\begin{enumerate}[\rm(1)]
    \item 
    $\curv_R(k)\le \sup\left\{\frac{\mu(M)}{e(M)-\mu(M)}, \injcurv_R(M)\right\}$ provided $M$ is not Ulrich.
    \item 
    $\curv_R(k)\le \sup\left\{\frac{\type(M)}{e(M)-\type(M)}, \curv_R(M)\right\}$ provided $M$ is not Ulrich.
    \item
    If $e(M)\ge 2\mu(M)$, then $\cx_R(k)\le 1+\injcx_R(M)$, i.e., $\cx_R(k)\in\{\injcx_R(M),1+\injcx_R(M)\}$.
    \item 
    If $e(M)>2\mu(M)$, then $\cx_R(k)=\injcx_R(M)$ and $\curv_R(k)=\injcurv_R(M)$.
    \item
    If $e(M)\ge 2\type(M)$, then $\cx_R(k)\le 1+\cx_R(M)$, i.e., $\cx_R(k)\in\{\cx_R(M),1+\cx_R(M)\}$.
    \item 
    If $e(M)>2\type(M)$, then $\cx_R(k)=\cx_R(M)$ and $\curv_R(k)=\curv_R(M)$.
\end{enumerate}
\end{corollary}

\begin{remark}\label{rmk:max-cx}
    In the context of (4) and (6) in both Corollaries~\ref{cor:cx-curv-k-CM-mod} and \ref{cor:main-cx-curv-CM-min-mult}, we should note that the following classes of modules have maximal complexity: (1)~Homomorphic images of finite direct sums of syzygies of the residue field \cite[Cor.~9]{Avr96}. (2)~Ulrich modules over CM local rings \cite[Cor.~3.8]{DG23}, (3)~$\fm$-primary integrally closed ideals \cite[Thm.~2.7]{GP23}, (4)~Integrally closed ideals $I$ of $R$ such that $\depth(R/I)=0$ \cite[Cor.~5.5.(1)]{GS22}. (5)~Burch submodules of a module \cite[Rmk.~3.10]{DK22}. Note that (3) is included in (4), and (4) is included in (5).
\end{remark}

\begin{remark}\label{rmk:exam-sec-6}
We provide \Cref{exam-3-cx-curv-inequality} which shows that the hypothesis of minimal multiplicity cannot be dropped from \Cref{cor:main-cx-curv-CM-min-mult} (thus from \Cref{thm:cx-curv-min-mult}). \Cref{exam-1-cx-curv-inequality} ensures that the inequalities in \Cref{cor:main-cx-curv-CM-min-mult}.(1) and (2) (hence in \Cref{thm:cx-curv-min-mult}) are sharp. The conditions `$e(M) \ge 2 \mu(M)$' and `$e(M) \ge 2 \type(M)$' cannot be omitted from \Cref{cor:main-cx-curv-CM-min-mult}.(3) and (5) respectively. Moreover, the conditions `$e(M) > 2 \mu(M)$' and `$e(M) > 2 \type(M)$' in \Cref{cor:main-cx-curv-CM-min-mult}.(4) and (6) (hence in \Cref{thm:cx-curv-min-mult}) cannot be replaced by `$e(M) \ge 2 \mu(M)$' and `$e(M) \ge 2 \type(M)$' respectively as shown in \Cref{exam:strict-ineq-cant-be-eq}.
\end{remark}

\begin{example}\label{exam-3-cx-curv-inequality}
Consider $R$ to be Gorenstein but not  abstract hypersurface (i.e., its codimension $>1$). Then, $e(R)\ge \mu(\fm)-\dim(R)+1>2=2\mu(R)=2\type(R)$. As $\Ext_R^{\gg 0}(k,R)=0$, one has that $\injcurv_R(R)=\injcx_R(R)=0$. If the conclusion of Corollary~\ref{cor:main-cx-curv-CM-min-mult}.(1) or (2) held true with $M=R$, then we would have that $\curv(k)<1$, and hence by \Cref{rmk:cx-curv}.(\ref{finite-pd-cx-curv}), $\curv(k)=0$, equivalently, $\pd_R(k)<\infty$, i.e., $R$ is regular, contradiction! Similarly, if the conclusion of Corollary~\ref{cor:main-cx-curv-CM-min-mult}.(3) or (5) held true with $M=R$, then we would have that $\cx(k)\le 1$, contradicting $R$ is not an abstract hypersurface.
\end{example}

As a consequence of \Cref{cor:main-cx-curv-CM-min-mult}, we show that every nonzero module $M$ of minimal multiplicity, and satisfying $ e(M) \ge 2 \mu(M) $ or $e(M)\ge 2\type(M) $ can be used to characterize complete intersection rings.

\begin{corollary}\label{cor:char-CI-minmult-rings}
    Let $M$ be a nonzero $R$-module of minimal multiplicity.
    \begin{enumerate}[\rm (1)]
        \item Let $ e(M) \ge 2 \mu(M) $. Then the following are equivalent:\\
        $\rm (a)$~$R$ is complete intersection; $\rm (b)$~$\injcx_R(M)<\infty$; and $\rm (c)$~$\injcurv_R(M)\le 1$.
        \item Let $e(M)\ge 2\type(M) $. Then the following are equivalent:\\
        $\rm (a)$~$R$ is complete intersection; $\rm (b)$~$\cx_R(M)<\infty$; and $\rm (c)$~$\curv_R(M)\le 1$.
    \end{enumerate}
\end{corollary}

\begin{proof}
    The forward implications in (1) and (2) are clear due to \Cref{prop:cx-curv-facts}.(3) and \Cref{rmk:cx-curv}.\eqref{finite-cx-implies-curv}. It remains to prove (c) $\Rightarrow$ (a) in (1), and (c) $\Rightarrow$ (a) in (2). 
   The hypothesis of (1) and (2) yield that $M$ cannot be Ulrich (cf.~\Cref{thm:e-mu-type}). Thus, the  implication (c) $\Rightarrow$ (a) in (1) and (2) follow from \Cref{cor:main-cx-curv-CM-min-mult}.(1) and (2) respectively, and \Cref{prop:cx-curv-facts}.(2).  
\end{proof}

For a large class of modules, the condition on annihilation of Ext or Tor modules by some power of the maximal ideal as in \Cref{thm:cx-curv-min-mult} and \Cref{Cor:cx-curv-min-CM} can be dropped.

\begin{corollary}\label{cor:loc-fin-pd-min}
Let $M$ and $X$ be nonzero $R$-modules such that $M$ has minimal multiplicity.
\begin{enumerate}[\rm(1)]
    \item 
    Suppose $X$ is locally of finite projective dimension on $\Spec(R) \smallsetminus \{ \fm \}$. Then
    \begin{enumerate}[\rm(i)]
        \item $\curv_R(X)\le \sup\left\{\frac{\mu(M)}{e(M)-\mu(M)}, \curv_R(X,M)\right\}$ provided $M$ is not Ulrich.
        \item If $e(M)\ge 2\mu(M)$, then $\cx_R(X)\le 1+\cx_R(X,M)$.
        \item If $e(M)>2\mu(M)$, then $\cx_R(X)\le \cx_R(X,M)$ and $\curv_R(X)\le \curv_R(X,M)$. 
    \end{enumerate}
    \item 
    Suppose $M$ is locally of finite projective dimension on $\Spec(R) \smallsetminus \{ \fm \}$. Then
    \begin{enumerate}[\rm (i)]
        \item $\injcurv_R(X) \le \sup\left\{ \frac{e(M)}{\mu(M)}-1, \curv_R(M,X)\right\}$.
        \item If $e(M)\le 2\mu(M) $, then $\injcx_R(X)\le 1+\cx_R(M,X)$.
        \item If $e(M)<2\mu(M) $, then $\injcx_R(X)\le \cx_R(M,X)$ and  $\injcurv_R(X)\le \curv_R(M,X)$.
    \end{enumerate}
    \item 
    Suppose either $M$ or $X$ is locally of finite projective dimension on $\Spec(R) \smallsetminus \{ \fm \}$. Then
     \begin{enumerate}[\rm(i)]
        \item $\curv_R(X)\le \sup\left\{\frac{\type(M)}{e(M)-\type(M)}, \tcurv_R(X,M)\right\}$ provided $M$ is not Ulrich.
        \item If $e(M)\ge 2\type(M)$, then $\cx_R(X)\le 1+\tcx_R(X,M)$.
        \item If $e(M)>2\type(M)$, then $\cx_R(X)\le \tcx_R(X,M)$ and $\curv_R(X)\le \tcurv_R(X,M)$. 
    \end{enumerate}
    \item 
    Suppose either $X$ is CM, or $R$ is CM. Also assume that $M$ is locally of finite projective dimension on $\Spec(R) \smallsetminus \{ \fm \}$. Then
    \begin{enumerate}[\rm(i)]
        \item $\injcurv_R(X)\le \sup\left\{\frac{\type(M)}{e(M)-\type(M)}, \curv_R(M,X)\right\}$ provided $M$ is not Ulrich.
        \item If $e(M)\ge 2\type(M)$, then $\injcx_R(X)\le 1+\cx_R(M,X)$.
        \item If $e(M)>2\type(M)$, then $\injcx_R(X)\le \cx_R(M,X)$ and $\injcurv_R(X)\le \curv_R(M,X)$. 
    \end{enumerate} 
\end{enumerate}    
\end{corollary}

\begin{proof}
    (1) In view of Lemma~\ref{lem:Tor-m-zero}.(1), there exists $g>0$ such that $\fm^g \Ext^n_R(X,M)=0$ for all $n>\dim(R)$. Hence the desired inequalities follow from Theorem~\ref{thm:cx-curv-min-mult}.(1).

    (2) In view of Lemma~\ref{lem:Tor-m-zero}.(1), there exists $f>0$ such that $\fm^f \Ext^n_R(M,X)=0$ for all $n>\dim(R)$. Hence we are done by Theorem~\ref{thm:cx-curv-min-mult}.(2).

    (3) As above the results follow from Lemma~\ref{lem:Tor-m-zero}.(2) and Theorem~\ref{thm:cx-curv-min-mult}.(3).

    (4) By Lemma~\ref{lem:Tor-m-zero}.(1), there exists $f>0$ such that $\fm^f \Ext^n_R(M,X)=0$ for all $n>\dim(R)$. So, when $X$ is CM, we are done by Theorem~\ref{thm:cx-curv-min-mult}.(4), and when $R$ is CM, we need to use Corollary~\ref{Cor:cx-curv-min-CM}.
\end{proof}

\begin{remark}
    The CM assumption on $X$ or $R$ in \Cref{cor:loc-fin-pd}.(3) can be removed under the assumption that $M$ has minimal multiplicity, as noted in \Cref{cor:loc-fin-pd-min}.(2).
\end{remark}

We end this section by showing that  modules of minimal multiplicity have trivial Ext vanishing. 

\begin{theorem}\label{thm:trivial-Ext-van}
Let $M$ be a nonzero $R$-module of minimal multiplicity.
\begin{enumerate}[\rm (1)]
    \item If there exists an integer $h\ge 0$ such that $\fm^h\Ext_R^{\gg 0}(M,M)=0$ $($e.g., $M$ is locally of finite projective dimension on the punctured spectrum, {\rm cf.~\Cref{lem:Tor-m-zero}}$)$, then
    $$\inf\{\cx_R(M),\injcx_R(M)\}\le \cx_R(M,M) \text{ and } \inf\{\curv_R(M), \injcurv_R(M)\}\le  \curv_R(M,M)$$ 
    \item In particular, if $\Ext_R^{\gg 0}(M,M)=0$, then either $\pd_R(M)<\infty$ or $\id_R(M)<\infty$.
    \item If $R$ is Gorenstein and $\Ext_R^{\gg 0}(M,M)=0$, then $R$ is hypersurface.
\end{enumerate}
\end{theorem}

\begin{proof}
(1) First we prove the inequality about complexities:
Three cases may appear. First, assume that $e(M)<2\mu(M)$. Then \Cref{thm:main-cx-curv-CM-mod}.(3).(iii) yields that $\injcx_R(M) \le \cx_R(M,M)$.  In the second case, assume that $e(M) > 2\mu(M)$. Then \Cref{thm:cx-curv-min-mult}.(1).(iii) provides that $\cx_R(M)\le \cx_R(M,M)$. This shows that if $e(M)\neq 2\mu(M)$, then $\inf\{\cx_R(M),\injcx_R(M)\}\le \cx_R(M,M)$. Finally, assume that $e(M)=2\mu(M)$. In order to prove the desired inequality, we may assume that $\cx_R(M,M)<\infty$. Then, in view of \Cref{thm:main-cx-curv-CM-mod}.(3).(ii), $\injcx_R(M)<\infty$. Therefore, by \Cref{cor:main-cx-curv-CM-min-mult}.(3), $\cx_R(k)<\infty$, and hence $R$ is complete intersection (cf.~\Cref{prop:cx-curv-facts}.(2)). Now \cite[Theorem~II.(1)]{AB00} implies that $\cx_R(M)=\injcx_R(M)=\cx_R(M,M)$. This finishes all the cases.

Now we prove the case for curvature: If $e(M)< 2\mu(M)$, then by \Cref{thm:main-cx-curv-CM-mod}.(3).(iii), $\injcurv_R(M)\le  \curv_R(M,M)$. If $e(M)> 2\mu(M)$, then $M$ is not Ulrich, and hence \Cref{thm:cx-curv-min-mult}.(1).(iii) yields that $\curv_R(M)\le \curv_R(M,M)$. Now assume that $e(M)=2\mu(M)$. If $\curv_R(M,M)\ge 1$, then in view of \Cref{thm:main-cx-curv-CM-mod}.(3).(i), $\injcurv_R(M)\le \sup\{1,\curv_R(M,M)\}=\curv_R(M,M)$. So it remains to consider the case $\curv_R(M,M)<1$. In this case, by \Cref{rmk:cx-curv}.(5).(a), $\cx_R(M,M)=0$, which implies that either $\cx_R(M)=0$ or $\injcx_R(M)=0$ by the complexity case already proved, and hence either $\curv_R(M)=0$ or $\injcurv_R(M)=0$. This enumerates all the cases.

(2) Suppose $\Ext_R^{\gg 0}(M,M)=0$. Then $\cx_R(M,M)=0$. Hence, by the first part, either $\cx_R(M)=0$ or $\injcx_R(M)=0$, which implies that either $\pd_R(M)<\infty$ or $\id_R(M)<\infty$.

(3) Since $R$ is Gorenstein, by (2), both $\pd_R(M)$ and $\id_R(M)$ are finite, i.e., $\cx_R(M)=0=\injcx_R(M)$. Therefore, depending on whether $e(M)\le 2\mu(M)$ or $e(M)\ge 2\mu(M)$, in view of Corollaries~\ref{cor:cx-curv-k-CM-mod}.(3) or \ref{cor:main-cx-curv-CM-min-mult}.(3) respectively, one gets that $\cx_R(k)\le 1$. Hence, by \Cref{prop:cx-curv-facts}.(2), $R$ is hypersurface. 
\end{proof}

\begin{remark}
    It follows from \cite[Prop.~2.5]{DG23} that Ulrich modules have trivial Ext vanishing. In view of \Cref{rmk:Ulrich-min-mult}, note that \Cref{thm:trivial-Ext-van}.(2) strengthens this fact.
\end{remark}

\section{Some applications to characterizations of local rings}\label{sec:char}

In this section, we obtain various criteria for a local ring to be complete intersection (including hypersurface and regular local ring) in terms of complexity and curvature of pair of certain CM modules. These criteria can be regarded as applications of Theorems~\ref{thm:main-cx-curv-CM-mod} and \ref{thm:cx-curv-min-mult}.

\begin{theorem}\label{thm:cx-M-N-inequa}
    Let $M$ and $N$ be nonzero CM $R$-modules such that $\fm^h\Ext_R^{\gg 0}(M,N)=0$ for some $h\ge 0$. Furthermore, assume that at least one of the following conditions holds.
    \begin{enumerate}[\rm (1)]
        \item 
        $ e(M) \le 2 \mu(M) $ and $ e(N) \le 2 \type(N) $.
        \item
        $ e(M) \le 2 \mu(M)$, $e(N) \ge 2 \mu(N) $, and $N$ has minimal multiplicity.
        \item 
        $ e(M) \ge 2 \type(M)$, $M$ has minimal multiplicity, and $e(N) \le 2 \type(N) $.
        \item
        $ e(M) \ge 2 \type(M)$, $e(N) \ge 2 \mu(N) $, and both $M$ and $N$ have minimal multiplicity.  
    \end{enumerate}

    Then, $\sup\big\{\cx_R(M),\injcx_R(M)\big\}+\sup\big\{\injcx_R(N),\cx_R(N)\big\} \le 2\big(1+\cx_R(M,N)\big)$. Moreover, the following statements are equivalent:
    \begin{enumerate}[\rm (a)]
        \item $\cx_R(M,N)<\infty$.
        \item $\curv_R(M,N)\le 1$.
        \item $R$ is complete intersection $($of codimension at most $2+\cx_R(M,N)$$)$.
    \end{enumerate}
\end{theorem}

\begin{proof}
We first prove the second assertion. Note that (a) $\Rightarrow$ (b) by \Cref{rmk:cx-curv}.(\ref{finite-cx-implies-curv}), and (c) $\Rightarrow$ (a) by \Cref{prop:cx-curv-facts}.(\ref{CI-ring-finite-cx-pair}). In fact, these two implications hold true for any pair of modules. Thus it is enough to show that (b) $\Rightarrow$ (c) under each of the given conditions. For the first assertion, if $\cx_R(M,N)=\infty$, then there is nothing to prove. So we may assume that $\cx_R(M,N)<\infty$, and establish the desired inequality under each of the given conditions.

(1) Let $\curv_R(M,N)\le 1$. Since $ e(N) \le 2 \type(N) $, one has that $\frac{e(N)}{\type(N)}-1\le 1$. So \Cref{thm:main-cx-curv-CM-mod}.(2).(i) yields that $\curv_R(M)\le 1$. Therefore, by \Cref{cor:char-CI-rings}.(1), $R$ is complete intersection. Hence, by \Cref{prop:cx-curv-facts}.\ref{CI-ring-finite-cx-pair}, $\cx_R(M,N)<\infty$. In view of \Cref{thm:main-cx-curv-CM-mod}.(2).(ii), $\cx_R(M)\le 1+\cx_R(M,N)$. Consequently, \Cref{cor:cx-curv-k-CM-mod}.(3) yields that $\cx_R(k)\le 1+\cx_R(M)\le 2+\cx_R(M,N)<\infty$. So, by \Cref{prop:cx-curv-facts}.(2), $R$ has codimension at most $2+\cx_R(M,N)$. This proves (b) $\Rightarrow$ (c). For the desired inequality, assume that $\cx_R(M,N)<\infty$. Then, by the equivalence of (a) and (c), $R$ is complete intersection of codimension at most $2+\cx_R(M,N)$. Hence, in view of \cite[Thm.~II.(1)]{AB00}, $\cx_R(M)=\injcx_R(M)$ and $\cx_R(N)=\injcx_R(N)$. Moreover, by \cite[Thm.~II.(3)]{AB00}, $\cx_R(M)+\cx_R(N) \le \codim(R)+\cx_R(M,N) \le 2+2\cx_R(M,N)$. Thus the desired inequality follows.

(2) Let $\curv_R(M,N)\le 1$. Since $e(N) \ge 2 \mu(N) $, one has that $ \frac{\mu(N)}{e(N)-\mu(N)} \le 1 $, and $N$ is not Ulrich (cf.~\Cref{defn:Ulrich-mod}). So \Cref{thm:cx-curv-min-mult}.(1).(i) yields that $\curv_R(M)\le 1$. Therefore, by \Cref{cor:char-CI-rings}.(1), $R$ is complete intersection. Hence, by \Cref{prop:cx-curv-facts}.\ref{CI-ring-finite-cx-pair}, $\cx_R(M,N)<\infty$. In view of \Cref{thm:main-cx-curv-CM-mod}.(3).(ii), $\injcx_R(N)\le 1+\cx_R(M,N)$. Consequently, \Cref{cor:main-cx-curv-CM-min-mult}.(3) yields that $ \cx_R(k) \le 1+\injcx_R(N) \le 2+\cx_R(M,N) < \infty $. So, by \Cref{prop:cx-curv-facts}.(2), $R$ has codimension at most $2+\cx_R(M,N)$. This proves (b) $\Rightarrow$ (c). The rest of the proof is similar as that of (1).

(3) Let $\curv_R(M,N)\le 1$. Since $ e(M) \ge 2 \type(M) $, one has that $\frac{\type(M)}{e(M)-\type(M)}\le 1$, and $M$ is not Ulrich (by \Cref{thm:e-mu-type}). So \Cref{thm:cx-curv-min-mult}.(4).(i) yields that $\injcurv_R(N)\le 1$. Therefore, by \Cref{cor:char-CI-rings}.(2), $R$ is complete intersection. Hence, by \Cref{prop:cx-curv-facts}.\ref{CI-ring-finite-cx-pair}, $\cx_R(M,N)<\infty$. By \Cref{thm:cx-curv-min-mult}.(4).(ii), $\injcx_R(N) \le 1 + \cx_R(M,N)$. Hence \Cref{cor:cx-curv-k-CM-mod}.(5) yields that $ \cx_R(k) \le 1+\injcx_R(N) \le 2+\cx_R(M,N) < \infty $.
The rest of the proof is similar as that of (1).

(4) Let $\curv_R(M,N)\le 1$. Since $e(N) \ge 2 \mu(N) $, one has that $ \frac{\mu(N)}{e(N)-\mu(N)} \le 1 $, and $N$ is not Ulrich. So \Cref{thm:cx-curv-min-mult}.(1).(i) yields that $\curv_R(M)\le 1$. Therefore, by \Cref{cor:char-CI-minmult-rings}.(2), $R$ is complete intersection. Hence, by \Cref{prop:cx-curv-facts}.\ref{CI-ring-finite-cx-pair}, $\cx_R(M,N)<\infty$. By \Cref{thm:cx-curv-min-mult}.(4).(ii), $\injcx_R(N) \le 1+ \cx_R (M,N)$. Hence \Cref{cor:main-cx-curv-CM-min-mult}.(3) yields that $ \cx_R(k) \le 1+\injcx_R(N) \le 2+\cx_R(M,N) < \infty $.
The rest of the proof is similar as that of (1).
\end{proof}

Since $\Ext_R^{\gg 0}(M,N)=0$ \iff $\cx_R(M,N)=0$, the following is a consequence of \Cref{thm:cx-M-N-inequa}.

\begin{corollary}\label{cor:cx-M-N-inequa}
    Let $M$ and $N$ be nonzero CM $R$-modules such that $\Ext_R^{\gg 0}(M,N)=0$, and at least one of the four conditions {\rm (1)-(4)} in {\rm \Cref{thm:cx-M-N-inequa}} holds.
    Then, $R$ is complete intersection, and $\codim(R) \le 2$.
\end{corollary}

When exactly one of the two inequalities in each condition of \Cref{thm:cx-M-N-inequa} is a strict inequality, then the conclusion of \Cref{thm:cx-M-N-inequa} can be improved as follows.

\begin{theorem}\label{thm:cx-M-N-inequa-2}
    Let $M$ and $N$ be nonzero CM $R$-modules such that $\fm^h\Ext_R^{\gg 0}(M,N)=0$ for some $h\ge 0$. Furthermore, assume that at least one of the following conditions holds.
    \begin{enumerate}[\rm (4A)]
        \item[\rm (1A)]
        $ e(M) \le 2 \mu(M) $ and $ e(N) < 2 \type(N) $.
        \item[\rm (1B)]
        $ e(M) < 2 \mu(M) $ and $ e(N) \le 2 \type(N) $.
        \item[\rm (2A)]
        $ e(M) < 2 \mu(M)$, $e(N) \ge 2 \mu(N) $, and $N$ has minimal multiplicity.
        \item[\rm (2B)]
        $ e(M) \le 2 \mu(M)$, $e(N) > 2 \mu(N) $, and $N$ has minimal multiplicity.
        \item[\rm (3A)]
        $ e(M) > 2 \type(M)$, $M$ has minimal multiplicity, and $e(N) \le 2 \type(N) $.
        \item[\rm (3B)]
        $ e(M) \ge 2 \type(M)$, $M$ has minimal multiplicity, and $e(N) < 2 \type(N) $.
        \item[\rm (4A)]
        $ e(M) > 2 \type(M)$, $e(N) \ge 2 \mu(N) $, and both $M$ and $N$ have minimal multiplicity.
        \item[\rm (4B)]
        $ e(M) \ge 2 \type(M)$, $e(N) > 2 \mu(N) $, and both $M$ and $N$ have minimal multiplicity.
    \end{enumerate}

    Then, $\sup\big\{\cx_R(M),\injcx_R(M)\big\}+\sup\big\{\injcx_R(N),\cx_R(N)\big\} \le \big(1+2 \cx_R(M,N)\big)$. Moreover, the following statements are equivalent:
    \begin{enumerate}[\rm (a)]
        \item $\cx_R(M,N)<\infty$.
        \item $\curv_R(M,N)\le 1$.
        \item $R$ is complete intersection $($of codimension at most $1+\cx_R(M,N)$$)$.
    \end{enumerate}
\end{theorem}

\begin{proof}
    The proof goes along the same lines as that of \Cref{thm:cx-M-N-inequa} except the following changes.
    
    (1A) Use \Cref{thm:main-cx-curv-CM-mod}.(2).(iii) instead of \Cref{thm:main-cx-curv-CM-mod}.(2).(ii) in the proof of \Cref{thm:cx-M-N-inequa}.(1).

    (1B) Use \Cref{thm:main-cx-curv-CM-mod}.(3).(iii) and \Cref{cor:cx-curv-k-CM-mod}.(5)  instead of \Cref{thm:main-cx-curv-CM-mod}.(2).(ii) and \Cref{cor:cx-curv-k-CM-mod}.(3) respectively in the proof of \Cref{thm:cx-M-N-inequa}.(1).

    (2A) Use \Cref{thm:main-cx-curv-CM-mod}.(3).(iii) instead of \Cref{thm:main-cx-curv-CM-mod}.(3).(ii) in the proof of \Cref{thm:cx-M-N-inequa}.(2).

    (2B) Use \Cref{thm:cx-curv-min-mult}.(1).(iii) and \Cref{cor:cx-curv-k-CM-mod}.(3) instead of \Cref{thm:main-cx-curv-CM-mod}.(3).(ii) and \Cref{cor:main-cx-curv-CM-min-mult}.(3) respectively in the proof of \Cref{thm:cx-M-N-inequa}.(2).

    (3A) Use \Cref{thm:cx-curv-min-mult}.(4).(iii) instead of \Cref{thm:cx-curv-min-mult}.(4).(ii) in the proof of \Cref{thm:cx-M-N-inequa}.(3).

    (3B)  Use \Cref{thm:main-cx-curv-CM-mod}.(2).(iii) and \Cref{cor:main-cx-curv-CM-min-mult}.(5) instead of \Cref{thm:cx-curv-min-mult}.(4).(ii) and \Cref{cor:main-cx-curv-CM-min-mult}.(3) in the proof of \Cref{thm:cx-M-N-inequa}.(3).

    (4A) Use \Cref{thm:cx-curv-min-mult}.(4).(iii) instead of \Cref{thm:cx-curv-min-mult}.(4).(ii) in the proof of \Cref{thm:cx-M-N-inequa}.(4).

    (4B) Use \Cref{thm:cx-curv-min-mult}.(1).(iii) and \Cref{cor:main-cx-curv-CM-min-mult}.(5) instead of \Cref{thm:cx-curv-min-mult}.(4).(ii) and \Cref{cor:main-cx-curv-CM-min-mult}.(3) in the proof of \Cref{thm:cx-M-N-inequa}.(4).
\end{proof}

As a consequence of \Cref{thm:cx-M-N-inequa-2}, one obtains certain criteria for hypersurface rings.

\begin{corollary}\label{cor:cx-M-N-inequa-2}
    Let $M$ and $N$ be nonzero CM $R$-modules such that $\Ext_R^{\gg 0}(M,N)=0$, and at least one of the eight conditions in {\rm \Cref{thm:cx-M-N-inequa-2}} holds.
    Then, $R$ is a hypersurface ring.
\end{corollary}

When both the inequalities in each of the four conditions of \Cref{thm:cx-M-N-inequa} are strict inequalities, then the conclusion of \Cref{thm:cx-M-N-inequa} can be improved further.

\begin{theorem}\label{thm:cx-M-N-inequa-3}
    Let $M$ and $N$ be nonzero CM $R$-modules such that $\fm^h\Ext_R^{\gg 0}(M,N)=0$ for some $h\ge 0$. Furthermore, assume that at least one of the following conditions holds.
    \begin{enumerate}[\rm (1)]
        \item 
        $ e(M) < 2 \mu(M) $ and $ e(N) < 2 \type(N) $.
        \item
        $ e(M) < 2 \mu(M)$, $e(N) > 2 \mu(N) $, and $N$ has minimal multiplicity.
        \item 
        $ e(M) > 2 \type(M)$, $M$ has minimal multiplicity, and $e(N) < 2 \type(N) $.
        \item
        $ e(M) > 2 \type(M)$, $e(N) > 2 \mu(N) $, and both $M$ and $N$ have minimal multiplicity.
    \end{enumerate}

    Then, $\sup\big\{\cx_R(M),\injcx_R(M)\big\}+\sup\big\{\injcx_R(N),\cx_R(N)\big\} \le \big(2 \cx_R(M,N)\big)$. Moreover, the following statements are equivalent:
    \begin{enumerate}[\rm (a)]
        \item $\cx_R(M,N)<\infty$.
        \item $\curv_R(M,N)\le 1$.
        \item $R$ is complete intersection $($of codimension at most $\cx_R(M,N)$$)$.
    \end{enumerate}
\end{theorem}

\begin{proof}
The proof goes along the same lines as that of \Cref{thm:cx-M-N-inequa} except the following changes.

    (1) Use \Cref{thm:main-cx-curv-CM-mod}.(2).(iii) and \Cref{cor:cx-curv-k-CM-mod}.(4) instead of \Cref{thm:main-cx-curv-CM-mod}.(2).(ii) and \Cref{cor:cx-curv-k-CM-mod}.(3) respectively in the  proof of \Cref{thm:cx-M-N-inequa}.(1).

    (2) Use \Cref{thm:main-cx-curv-CM-mod}.(3).(iii) and \Cref{cor:main-cx-curv-CM-min-mult}.(4) instead of \Cref{thm:main-cx-curv-CM-mod}.(3).(ii) and \Cref{cor:main-cx-curv-CM-min-mult}.(3) respectively in the  proof of \Cref{thm:cx-M-N-inequa}.(2).

    (3) Use \Cref{thm:cx-curv-min-mult}.(4).(iii) and \Cref{cor:cx-curv-k-CM-mod}.(6) instead of \Cref{thm:cx-curv-min-mult}.(4).(ii) and \Cref{cor:cx-curv-k-CM-mod}.(5) respectively in the  proof of \Cref{thm:cx-M-N-inequa}.(3).

    (4) Use \Cref{thm:cx-curv-min-mult}.(4).(iii) and \Cref{cor:main-cx-curv-CM-min-mult}.(4) instead of \Cref{thm:cx-curv-min-mult}.(4).(ii) and \Cref{cor:main-cx-curv-CM-min-mult}.(3) respectively in the  proof of \Cref{thm:cx-M-N-inequa}.(4).
\end{proof}

\Cref{thm:cx-M-N-inequa-3} particularly provides certain criteria for a local ring to be regular.

\begin{corollary}\label{cor:cx-M-N-inequa-3}
    Let $M$ and $N$ be nonzero CM $R$-modules such that $\Ext_R^{\gg 0}(M,N)=0$, and at least one of the four conditions {\rm (1)-(4)} in {\rm \Cref{thm:cx-M-N-inequa-3}} holds.
    Then, $R$ is regular.
\end{corollary}

We next observe that the injective complexity (resp., injective curvature) of any nonzero module of finite projective dimension and the complexity (resp., curvature) of any nonzero module of finite injective dimension are all same, and equals the injective complexity (resp., injective curvature) of the ring. 

\begin{proposition}\label{prop:pd-injcurv-max}
    Let $M$ and $N$ be nonzero $R$-modules.
    \begin{enumerate}[\rm (1)]
        \item 
        Let $\pd_R(M) < \infty$. Then $\injcx_R(M)=\injcx_R(R)$ and $\injcurv_R(M)=\injcurv_R(R)$.
        \item 
        Let $\id_R(N) < \infty$. Then $\cx_R(N)=\injcx_R(R)$ and $\curv_R(N)=\injcurv_R(R)$.
    \end{enumerate}
\end{proposition}

\begin{proof}
    (1) First we show that $\injcx_R(M) \le \injcx_R(R)$ and $\injcurv_R(M) \le \injcurv_R(R)$. We use induction on $\pd_R(M)$. If $\pd_R(M)=0$, then it is trivial. Since $\pd_R(\Omega_R^1(M))\le n-1$, by the induction hypothesis, $\injcx_R(\Omega_R^1(M)) \le \injcx_R(R)$ and $\injcurv_R(\Omega_R^1(M)) \le \injcurv_R(R)$. From the short exact sequence $0 \to \Omega_R^1(M) \to F_0 \to M \to 0$, it is clear that $\mu_R^i(M) \le \mu_R^i(F_0) + \mu_R^{i+1}(\Omega_R^1(M))$ for all $i \ge 1$. Therefore, by \Cref{rmk:cx-curv-ineq}.(1) and \Cref{lem:sum-curv-ieq}, $\injcx_R(M) \le \injcx_R(R)$ and $\injcurv_R(M) \le \injcurv_R(R)$.

    Next we show that $\injcx_R(R) \le \injcx_R(M)$ and $\injcurv_R(R)\le\injcurv_R(M)$. Notice from \cite[Cor.~4.3.(2)]{Fo77} that there is an equality of formal Laurent series $\sum_{n=0}^{\infty} \mu^n_R(M) t^n=\big(\sum_{n=0}^\infty \mu^n_R(R)t^n\big)\big(\sum_{i=0}^r \beta^R_i(M)t^{-i}\big)$, where $r:=\pd_R(M)$. Hence the desired inequalities can be obtained from \Cref{lem:curv-poly-ieq}.

    (2) Since $\id_R(N) < \infty$, by Bass's conjecture, $R$ is CM. Without loss of generality, we may pass to completion, and assume that $R$ has a canonical module $\omega$. Since $\Hom_R(\omega,\omega)\cong R$, by \Cref{cx-duality}.(2).(b) and (c), one has that $\cx_R(\omega)= \injcx_R(R)$ and $\curv_R(\omega)= \injcurv_R(R)$ respectively. Since $\id_R(N) < \infty$, by \cite[3.3.28.(b)]{BH98}, there exists a finite $\omega$-resolution of $N$, say $ 0 \to \omega^{n_r} \to \ldots \to \omega^{n_2}  \to \omega^{n_1} \to N \to 0$. We break this into short exact sequences. Suppose the first one is $0 \to \omega^{n_r} \to \omega^{n_{r-1}} \to N_{r-1} \to 0$. It induces that $\beta_i^R(N_{r-1}) \le \beta_{i-1}^R(\omega^{n_r})+ \beta_i^R(\omega^{n_{r-1}})$ for all $i \ge 1$. Hence, by \Cref{rmk:cx-curv-ineq} and \Cref{lem:sum-curv-ieq}, one obtains that $\cx_R(N_{r-1}) \le \cx_R(\omega)$ and $\curv_R(N_{r-1}) \le \curv_R(\omega)$. Consider the second short exact sequence $0 \to N_{r-1} \to \omega^{n_{r-2}} \to N_{r-2} \to 0$. By similar arguments, one gets that $\cx_R(N_{r-2}) \le \cx_R(\omega)$ and $\curv_R(N_{r-2}) \le \curv_R(\omega)$. Repeating the process, and using induction on $r$,
    we have $\cx_R(N) \le \cx_R(\omega)= \injcx_R(R)$ and $\curv_R(N) \le \curv_R(\omega)= \injcurv_R(R)$. For the other inequalities, in view of \cite[Cor.~4.3.(1)]{Fo77}, there is an equality of formal Laurent series $\sum_{n=0}^{\infty} \beta_n^R(N) t^n=\big(\sum_{n=0}^\infty \mu^n_R(R)t^n\big)\big(\sum_{i=0}^s \mu_R^i(N)t^{-i}\big)$, where $s:=\id_R(N)$. Therefore the desired inequalities follow from \Cref{lem:curv-poly-ieq}.
\end{proof}

The following is a direct consequence of \Cref{thm:cx-M-N-inequa}.(1) and \Cref{prop:pd-injcurv-max}, and can be regarded as a partial progress towards the exponential growth part of \cite[Ques.~1.3]{CSV} and \cite[Ques., p.~647]{JL07}.

\begin{corollary}\label{cor:exp-growth-JL}
\begin{enumerate}[\rm (1)]
\item
Let $R$ be a CM ring satisfying $e(R)\le 2\type(R)$. If there exists a nonzero CM $R$-module $M$ such that $e(M)\le 2\mu(M)$, $\fm^h\Ext^{\gg 0}_R(M,R)=0$ for some $h\ge 0$, and $\curv_R(M,R)\le 1$ $($in particular, if $\cx_R(M,R)<\infty$$)$, then $R$ is complete intersection. 
\item     
Suppose there exists nonzero $R$-modules $M_1$ and $M_2$, both of finite projective dimension, such that $M_1$ is CM, $e(M_1)\le 2\type(M_1)$, and $\injcurv_R(M_2)\le 1$. Then, $R$ is complete intersection.
   
Specifically, if $R$ is CM, $e(R)\le 2\type(R)$, and there exists a nonzero $R$-module $M$ of finite projective dimension such that $\injcurv_R(M)\le 1$,  then $R$ is complete intersection.
\item 
Suppose there exist nonzero $R$-modules $N_1$ and $N_2$, both of finite injective dimension, such that $N_1$ is CM, $e(N_1)\le 2\mu(N_1)$, and $\curv_R(N_2)\le 1$. Then, $R$ is complete intersection.
   
Specifically, if $e(R)\le 2\type(R)$ and there exists a nonzero $R$-module $N$ of finite injective dimension such that $\curv_R(N)\le 1$, then $R$ is complete intersection.
\end{enumerate}
\end{corollary}

\begin{proof}
    (1) It follows from \Cref{thm:cx-M-N-inequa}.(1) by taking $N:=R$.

    (2) By \Cref{prop:pd-injcurv-max}(1), $\injcurv_R(M_1)=\injcurv_R(R)=\injcurv_R(M_2)\le 1$. Therefore \Cref{thm:cx-M-N-inequa}.(1) applied to $M=k$ and $N=M_1$ implies that $R$ is complete intersection. The last part of (2) follows by taking $M_1:=R$ and $M_2:=M$.

    (3)  By \Cref{prop:pd-injcurv-max}.(2), $\curv_R(N_1)= \injcurv_R(R)= \curv_R(N_2) \le 1$. Therefore \Cref{thm:cx-M-N-inequa}.(1) applied to $M=N_1$ and $N=k$ implies that $R$ is complete intersection.
    
    For the last part, note that $R$ is CM by Bass's conjecture. Since $\id_R(N)<\infty$ and $\curv_R(N)\le 1$, in view of \Cref{prop:pd-injcurv-max}, one gets that $\injcurv_R(R)\le 1$. Hence, $R$ is complete intersection by (2).
\end{proof}

The first part of the following result can be regarded as a Tor version of \cite[Thm.~II.(3)]{AB00}, whereas parts (2) and (3) were recorded as Theorem 1.1 and Corollary 2.3 in \cite{Mi98} respectively. It is to be noted that a stronger version of part (2) of the following was obtained in \cite[Thm.~1.9]{HW97} simultaneously and independently of \cite{Mi98}. 

\begin{proposition}\label{prop:AB-like-result-on-tcx}
    Let $R$ be complete intersection. Let $M$ and $N$ be $R$-modules.
    \begin{enumerate}[\rm (1)]
        \item 
        If $\fm^h\Tor^R_{\gg 0}(M,N)=0$ for some $h\ge 0$, then
        $$\cx_R(M)+\cx_R(N)- \codim(R)\le \tcx_R(M,N)\le \min\{\cx_R(M),\cx_R(N)\}<\infty.$$
        \item 
        If $R$ is a hypersurface and $\Tor^R_{\gg 0}(M,N)=0$, then $M$ or $N$ has finite projective dimension.
        \item 
        If $\codim(R)=2$, $\Tor^R_{\gg 0}(M,N)=0$, and $M$ is not eventually periodic, then $\pd_R(N)<\infty$.
    \end{enumerate}
\end{proposition}

\begin{proof}
(1) Without loss of generality, we may pass to higher syzygies of $M, N$, and assume that $M, N$ are MCM $R$-modules. Note that $R$ itself is a canonical module of $R$. So, by virtue of \Cref{cx-duality}.(2), one obtains that $\cx_R(M,N^*)=\tcx_R(M,N)$ and $\cx_R(N)=\injcx_R(N^*)$. Since $R$ is complete intersection, in view of \cite[Thm.~II.(1)]{AB00}, $\injcx_R(N^*)=\cx_R(N^*)$. Thus $\cx_R(N)=\cx_R(N^*)$. Therefore, by \cite[Thm.~II.(3)]{AB00}, $\cx_R(M)+\cx_R(N) = \cx_R(M)+\cx_R(N^*) \le \codim(R)+\cx_R(M,N^*) = \codim(R)+\tcx_R(M,N)$, and $\tcx_R(M,N)=\cx_R(M,N^*)\le \min\{\cx_R(M), \cx_R(N^*)\}=\min\{\cx_R(M),\cx_R(N)\}<\infty$ by \Cref{prop:cx-curv-facts}.\eqref{CI-ring-finite-cx-curv}. 

(2) Since $R$ is a hypersurface and $\Tor^R_{\gg 0}(M,N)=0$, by (1), $\cx_R(M)+\cx_R(N)\le 1$. Therefore, since $\cx_R(M)$ and $\cx_R(N)$ are non-negative integers, it follows that at least one of them has to be zero, i.e., at least one of $M$ and $N$ has finite projective dimension. 

(3) Since $\codim R=2$ and $\Tor^R_{\gg 0}(M,N)=0$, by (1), $\cx_R(M)+\cx_R(N)\le 2$.  Since $M$ is not eventually periodic, by \cite[Thm.~4.4]{Avr89b}, $\cx_R(M)\ge 2$. It follows that $\cx_R(N)=0$, i.e., $\pd_R N<\infty$.
\end{proof}

Like \Cref{thm:cx-M-N-inequa}, we also obtain various criteria for complete intersection rings in terms of Tor complexity and curvature of pair of certain CM modules.

\begin{theorem}\label{thm:tcx-M-N-inequa}
    Let $M$ and $N$ be nonzero CM $R$-modules such that $\fm^h\Tor_{\gg 0}^R(M,N) = 0$ for some $h \ge 0$. Furthermore, assume that at least one of the following conditions holds.
    \begin{enumerate}[\rm (1)]
        \item 
        $ e(M) \le 2 \mu(M) $ and $ e(N) \le 2 \mu(N) $.
       \item 
       $ e(M) \ge 2 \type(M) $, $M$ has minimal multiplicity, and $ e(N) \le 2 \mu(N) $.
       \item
       $e(M) \ge 2 \type(M)$, $e(N)\ge 2\type(N)$, and both $M$ and $N$ have minimal multiplicity.
     \end{enumerate}

   Then, $\cx_R(M) + \cx_R(N) \le 2\big(1+\tcx_R(M,N)\big)$. Moreover, the following statements are equivalent:

   \begin{enumerate}[\rm (a)]
        \item $\tcx_R(M,N)<\infty$.
        \item $\tcurv_R(M,N)\le 1$.
        \item $R$ is complete intersection $($of codimension at most $2+\tcx_R(M,N)$$)$.
    \end{enumerate}
\end{theorem}

\begin{proof}
We first prove the second assertion. Note that  (a) $\Rightarrow$ (b) by \Cref{rmk:cx-curv}.\ref{finite-cx-implies-curv}, and (c) $\Rightarrow$ (a) by \Cref{prop:AB-like-result-on-tcx}.(1).
Thus it is enough to show that (b) $\Rightarrow$ (c) under each of the given conditions. For the first assertion, if $\tcx_R(M,N)=\infty$, then there is nothing to prove. So we may assume that $\tcx_R(M,N)<\infty$, and establish the desired inequality under each of the given conditions.


(1) Let $\tcurv_R(M,N) \le 1$. Since $e(M) \le 2 \mu(M)$, one has that $\frac{ e(M) - \mu(M)}{\mu(M)} \le 1$. So \Cref{thm:main-cx-curv-CM-mod}.(1).(i) yields that $\curv_R(N) \le 1$. Therefore, by \Cref{cor:char-CI-rings}.(1), $R$ is complete intersection. Hence, by \Cref{prop:AB-like-result-on-tcx}.(1), $\tcx_R(M,N)<\infty$. In view of \Cref{thm:main-cx-curv-CM-mod}.(1).(ii), $\cx_R(M) \le 1+\tcx_R(M,N)$. Consequently, \Cref{cor:cx-curv-k-CM-mod}.(3) yields that $\cx_R(k) \le 1+\cx_R(M) \le 2+\tcx_R(M,N) < \infty$. Therefore, by \Cref{prop:cx-curv-facts}.(2), $R$ has codimension at most $2+\tcx_R(M,N)$. This proves (b) $\Rightarrow$ (c). For the desired inequality, assume that $\tcx_R(M,N)<\infty$. Then, by the equivalence of (a) and (c), $R$ is complete intersection of codimension at most $2+\tcx_R(M,N)$. Hence, by \Cref{prop:AB-like-result-on-tcx}.(1), $\cx_R(M)+\cx_R(N) \le \codim(R)+\tcx_R(M,N) \le 2+2\tcx_R(M,N)$. Thus the desired inequality follows.


(2) Let $\tcurv_R(M,N) \le 1$. Since  $e(M) \ge 2 \type(M)$, one has that $\frac{\type(M)}{e(M) - \type(M)} \le 1$, and $M$ is not Ulrich (cf.~\Cref{thm:e-mu-type}). So \Cref{thm:cx-curv-min-mult}.(3).(i) yields that $\curv_R(N) \le 1$. Therefore, by \Cref{cor:char-CI-rings}.(1), $R$ is complete intersection. Hence, by \Cref{prop:AB-like-result-on-tcx}.(1), $\tcx_R(M,N)<\infty$. In view of \Cref{thm:cx-curv-min-mult}.(3).(ii), one has that $\cx_R(N) \le 1+\tcx_R(M,N)$. Hence \Cref{cor:cx-curv-k-CM-mod}.(3) yields that $\cx_R(k) \le 1+\cx_R(N) \le 2+\tcx_R(M,N) < \infty$. So $R$ is complete intersection of co-dimension at most $2+\tcx_R(M,N)$. The rest of the proof is similar as that of (1).

(3) Let $\tcurv_R(M,N) \le 1$. Then, as in the proof of (2), \Cref{thm:cx-curv-min-mult}.(3).(i) yields that $\curv_R(N) \le 1$. Therefore, by \Cref{cor:char-CI-minmult-rings}.(2), $R$ is complete intersection. Hence, by \Cref{prop:AB-like-result-on-tcx}.(1), $\tcx_R(M,N)<\infty$. In view of \Cref{thm:cx-curv-min-mult}.(3).(ii), one has that $\cx_R(N) \le 1+\tcx_R(M,N)$. Hence \Cref{cor:main-cx-curv-CM-min-mult}.(5) yields that $\cx_R(k) \le 1+\cx_R(N) \le 2+\tcx_R(M,N) < \infty$. So $R$ is complete intersection of co-dimension at most $2+\tcx_R(M,N)$. The rest of the proof is similar as that of (1).
\end{proof}

Since $\Tor^R_{\gg 0}(M,N)=0$ \iff $\tcx_R(M,N)=0$, the following is a consequence of \Cref{thm:tcx-M-N-inequa}.

\begin{corollary}\label{cor:tcx-M-N-inequa}
    Let $M$ and $N$ be nonzero CM $R$-modules such that $\Tor^R_{\gg 0}(M,N)=0$ and at least one of the three conditions {\rm (1)-(3)} in {\rm \Cref{thm:tcx-M-N-inequa}} holds. Then, $R$ is complete intersection and $\codim(R)\le 2$.
\end{corollary}

When exactly one of the two inequalities in each condition of \Cref{thm:tcx-M-N-inequa} is a strict inequality, then the conclusion of \Cref{thm:tcx-M-N-inequa} can be improved as follows.

\begin{theorem}\label{thm:tcx-M-N-inequa-2}
     Let $M$ and $N$ be nonzero CM $R$-modules such that $\fm^h\Tor_{\gg 0}^R(M,N) = 0$ for some $h \ge 0$. Furthermore, assume that at least one of the following conditions holds.
    \begin{enumerate}[\rm (1)]
       \item 
        $ e(M) \le 2 \mu(M) $, $ e(N) < 2 \mu(N) $.
       \item
         $e(M) > 2\type(M)$, and $M$ has minimal multiplicity, $e(N)\le 2\mu(N)$.
       \item 
        $e(M) \ge 2\type(M)$, and $M$ has minimal multiplicity, $e(N) < 2\mu(N)$.
       \item
       $e(M) > 2 \type(M)$, $e(N)\ge 2\type(N)$, and both $M$ and $N$ have minimal multiplicity.       
     \end{enumerate}
     
   Then, $\cx_R(M) + \cx_R(N) \le \big(1+2\tcx_R(M,N)\big)$. Moreover, the following statements are equivalent:
   \begin{enumerate}[\rm (a)]
        \item $\tcx_R(M,N)<\infty$.
        \item $\tcurv_R(M,N)\le 1$.
        \item $R$ is complete intersection $($of co-dimension at most $1+\tcx_R(M,N)$$)$.
    \end{enumerate}    
\end{theorem}

\begin{proof}
The proof goes along the same lines as that of \Cref{thm:tcx-M-N-inequa} except the following changes.

    (1) Use \Cref{thm:main-cx-curv-CM-mod}.(1).(iii) instead of \Cref{thm:main-cx-curv-CM-mod}.(1).(ii) in the proof of \Cref{thm:tcx-M-N-inequa}.(1).

    (2) Use \Cref{thm:cx-curv-min-mult}.(3).(iii) instead of \Cref{thm:cx-curv-min-mult}.(3).(ii) in the proof of \Cref{thm:tcx-M-N-inequa}.(2).

    (3) Use \Cref{cor:cx-curv-k-CM-mod}.(4) instead of \Cref{cor:cx-curv-k-CM-mod}.(3) in the proof of \Cref{thm:tcx-M-N-inequa}.(2).

    (4) Use \Cref{thm:cx-curv-min-mult}.(3).(iii) instead of \Cref{thm:cx-curv-min-mult}.(3).(ii) in the proof of \Cref{thm:tcx-M-N-inequa}.(3).
\end{proof}

As a consequence of \Cref{thm:tcx-M-N-inequa-2}, one obtains certain criteria for hypersurface rings.

\begin{corollary}\label{cor:tcx-M-N-inequa-2}
    Let $M$ and $N$ be nonzero CM $R$-modules such that $\Tor^R_{\gg 0}(M,N)=0$, and at least one of the four conditions {\rm (1)-(4)} in {\rm \Cref{thm:tcx-M-N-inequa-2}} holds. Then, $R$ is a hypersurface ring.
\end{corollary}

When both the inequalities in each of the three conditions of \Cref{thm:tcx-M-N-inequa} are strict inequalities, then the conclusion of \Cref{thm:tcx-M-N-inequa} can be improved further.

\begin{theorem}\label{thm:tcx-M-N-inequa-3}
     Let $M$ and $N$ be nonzero CM $R$-modules such that $\fm^h\Tor_{\gg 0}^R(M,N) = 0$ for some $h \ge 0$. Furthermore, assume that at least one of the following conditions holds.
    \begin{enumerate}[\rm (1)]
       \item 
        $ e(M) < 2 \mu(M) $ and $ e(N) < 2 \mu(N) $.
       \item 
       $ e(M) > 2 \type(M) $, $M$ has minimal multiplicity, and $ e(N) < 2 \mu(N) $.
       \item
       $e(M) > 2 \type(M)$, $e(N) > 2\type(N)$, and both $M$ and $N$ have minimal multiplicity.       
     \end{enumerate}

     Then, $\cx_R(M) + \cx_R(N) \le \big(2\tcx_R(M,N)\big)$. Moreover, the following statements are equivalent:
   \begin{enumerate}[\rm (a)]
        \item $\tcx_R(M,N)<\infty$.
        \item $\tcurv_R(M,N)\le 1$.
        \item $R$ is complete intersection $($of codimension at most $\tcx_R(M,N)$$)$.
    \end{enumerate}    
\end{theorem}

\begin{proof}
The proof goes along the same lines as that of \Cref{thm:tcx-M-N-inequa} except the following changes.

    (1) Use \Cref{thm:main-cx-curv-CM-mod}.(1).(iii) and \Cref{cor:cx-curv-k-CM-mod}.(4) instead of \Cref{thm:main-cx-curv-CM-mod}.(1).(ii) and  \Cref{cor:cx-curv-k-CM-mod}.(3) respectively in the proof of \Cref{thm:tcx-M-N-inequa}.(1).

    (2) Use \Cref{thm:cx-curv-min-mult}.(3).(iii) and \Cref{cor:cx-curv-k-CM-mod}.(4) instead of \Cref{thm:cx-curv-min-mult}.(3).(ii) and  \Cref{cor:cx-curv-k-CM-mod}.(3) respectively in the proof of \Cref{thm:tcx-M-N-inequa}.(2).

    (3) Use \Cref{thm:cx-curv-min-mult}.(3).(iii) and \Cref{cor:main-cx-curv-CM-min-mult}.(6) instead of \Cref{thm:cx-curv-min-mult}.(3).(ii) and \Cref{cor:main-cx-curv-CM-min-mult}.(5) respectively in the proof of \Cref{thm:tcx-M-N-inequa}.(3).  
\end{proof}

\Cref{thm:tcx-M-N-inequa-3} provides certain criteria for a local ring to be regular in terms of Tor vanishing.

\begin{corollary}\label{cor:tcx-M-N-inequa-3}
    Let $M$ and $N$ be nonzero CM $R$-modules such that $\Tor^R_{\gg 0}(M,N)=0$, and at least one of the three conditions {\rm (1)-(3)} in {\rm \Cref{thm:tcx-M-N-inequa-3}} holds.
    Then, $R$ is regular.
\end{corollary}

\section{Some applications to module of differentials}\label{sec:diff} 

In this section, we apply some of our results from the previous sections to give a few characterizations of regular local rings in terms of (injective) curvature of module of differentials under some additional hypotheses. Our results are motivated by a conjecture of Berger and a theorem due to Avramov-Herzog.

If $R$ is an algebra over a ring $A$, and $M$ is an $R$-module, then a map $\delta: R \to M$ is an $A$-derivation if it is an additive group homomorphism satisfying $\delta(xy)=x\delta(y)+y\delta(x)$ for all $x,y\in R$ and vanishes on the image of the structure map $A \to R$. The set of all $A$-derivations $R \to M$ forms an $R$-module $\Der_A(R, M)$.  When $M=R$, the notation $\Der_A(R, R)$ is simplified to $\Der_A(R)$. There exists a finitely generated $R$-module $\Omega_{R/A}$ (called the universally finite module of K\"{a}hler differentials of $R$ over $A$) and a universal derivation $d_{R/A}: R\to \Omega_{R/A}$ with the property that for any finitely generated $R$-module $M$, it induces an isomorphism $\Hom_R(\Omega_{R/A},M)\cong \Der_A(R,M)$ as $R$-modules.
When $R$ is essentially of finite type over a field $k$, the $R$-module $\Omega_{R/k}$ represents the functor $\Der_k(R,-)$ on the category of all $R$-modules. For further details and background, we refer the reader to \cite{kunz, herz, Mat86}.  



The following lemma, which will be needed to establish a fact for the rank of module of differentials, is a variation of \cite[Exer.~4.7.17.(a)]{BH98}. For an ideal $I$ and integer $n\ge 1$, $I^{(n)}$ denotes the $n$-th symbolic power of $I$. Recall that a ring $R$ is called equidimensional if $\dim(R/\fp)=\dim(R)$ for each $\fp\in \Min(R)$.


\begin{lemma}\label{8.3}
Let $A$ be a regular local ring, and $I$ be an ideal of $A$ such that $R:=A/I$ is reduced and equidimensional. Then, $I/I^{(2)}$ has rank as an $R$-module, and $\rank_R\left(I/I^{(2)}\right)=\htt(I)$. 
\end{lemma}

\begin{proof}  
First, we note that $\htt(I)=\htt(\fq)$ for each $\fq\in \Min(A/I)$. Indeed, since $A/I$ is equidimensional, it follows that $\dim(A/I)=\dim(A/\fq)$. Since $A$ is in particular CM, $\dim(A/I)=\dim(A)-\htt(I)$ and $\dim(A/\fq)=\dim(A)-\htt(\fq)$. Thus, $\htt(I)=\htt(\fq)$ for each $\fq\in \Min(A/I)$. Since $A/I$ is reduced, $I$ is a radical ideal of $A$. So, $I=\bigcap_{\fp\in \Min(A/I)} \fp$.  Then, $IA_{\fq}=\fq A_{\fq}$ and $I^{(2)}A_{\fq}=\fq^2A_{\fq}$ for all $\fq\in \Min(A/I)$. Since $A_{\fq}$ is regular, the maximal ideal $\fq A_{\fq}$ of $A_{\fq}$ is minimally generated by an $A_{\fq}$-regular sequence of length $\dim(A_{\fq}) = \htt(\fq) = \htt(I)$. Thus, $(I/I^{(2)})_{\fq}\cong I_{\fq}/I^{(2)}_{\fq}\cong \fq A_{\fq}/\fq^2A_{\fq}$ is a free module over $A_{\fq}/\fq A_{\fq}\cong R_{\fq}$ of rank $\dim A_{\fq}=\htt(I)$. This proves the claim.  
\end{proof}

The following fact about the rank of the module of K\"{a}hler differentials is folklore, however we add a proof for lack of a precise reference. 

\begin{lemma}\label{rank}
Suppose $R$ is an equidimensional reduced local ring, the residue field $k$ of $R$ has characteristic $0$, and $R$ is either essentially of finite type over $k$ or it is a homomorphic image of a power series ring $k[[X_1,\dots,X_n]]$ over $k$. Then, $\rank(\Omega_{R/k})=\dim(R)$. In particular, $e(\Omega_{R/k})=e(R)\dim(R)$. 
\end{lemma}

\begin{proof}
The case when $R$ is essentially of finite type over $k$ is standard, and follows from \cite[Cor.~4.22.(a), Prop.~5.7.(b) and Thm.~5.10.(c)]{kunz}. Indeed, for any $\fq\in \Min(R)$, $(\Omega_{R/k})_{\fq}\cong \Omega_{R_\fq/k}$ by \cite[Cor.~4.22.(a)]{kunz}. Note that $R_{\fq}$ is Artinian local and reduced, hence it is a field, and $R_{\fq}$ is the field of fractions of the domain $R/\fq$. Thus, by \cite[Thm~A.16]{BH98}, the transcendence degree of $R_{\fq}$ over $k$ equals $\dim(R/\fq)$, which is same as $\dim(R)$ since $R$ is equidimensional. Now, by \cite[Prop.~5.7.(b) and Thm.~5.10.(c)]{kunz}, $\Omega_{R_\fq/k}$ is a free $R_{\fq}$-module whose rank equals the transcendence degree of $R_{\fq}$ over $k$. This finally shows that for any $\fq\in \Min(R)$, $(\Omega_{R/k})_{\fq}$ is a free $R_{\fq}$-module of rank $\dim(R)$.

Next, for the case when $R \cong k[[X_1,\dots,X_n]]/I$ for some ideal $I$ of $k[[X_1,\cdots, X_n]]$ (where $n=\embdim(R)$), following the proof of \cite[Lem.~2.(i)]{Mir}, there is a short exact sequence $0\to I/I^{(2)}\to R^{\oplus n}\to \Omega_{R/k}\to 0$. Hence \Cref{8.3} yields that $\rank_R(\Omega_{R/k})=n-\htt(I)=\dim(R)$. 

We now prove the last part. If $R$ is Artinian, then $\Omega_{R/k}$ having rank $\dim(R)$ implies $\Omega_{R/k}=0$, and there is nothing to prove. So we may assume that $\dim(R)>0$.  Hence, by \cite[Cor.~4.7.9]{BH98}, it follows that $e(\Omega_{R/k})=e(R)\rank(\Omega_{R/k})=e(R)\dim(R)$.  
\end{proof}

In \cite[Thm.~2]{AH}, Avramov-Herzog proved that a graded algebra $R$ over a field $k$ of characteristic $0$ is complete intersection under the condition that $\curv_R(\Omega_{R/k})\le 1$. The following theorem can be regarded as a variation of this result. Since finite projective dimension of a module implies that module has curvature $0$, \Cref{thm:diff}.(1).(a) also provides partial progress towards \cite[Conj.~($\operatorname{C}_2$), p.~1807]{Va78}.

\begin{theorem}\label{thm:diff}
    Suppose $R$ is an equidimensional reduced local ring, the residue field $k$ of $R$ has characteristic $0$, and $R$ is either essentially of finite type over $k$ or it is a homomorphic image of a power series ring $k[[X_1,\dots,X_n]]$ over $k$. Let $\Omega_{R/k}$ be CM. Then, the following hold.
    \begin{enumerate}[\rm(1)]
        \item $R$ is regular under each of the following conditions.
    \begin{enumerate}[\rm (a)]
        \item 
        $e(R) \dim(R) \le 2 \embdim(R)$ and $\curv_R(\Omega_{R/k})\le 1$.
        \item 
        $R$ is Gorenstein, $e(R) \dim(R) \le 2 \embdim(R)$ and $\injcurv_R(\Der_k(R))\le 1$.
    \end{enumerate}
        \item 
        Suppose $R$ is not a field, $\Omega_{R/k}$ has minimal multiplicity, and $e(R)\dim(R)\ge 2\embdim(R)$. Then $\injcurv_R(\Omega_{R/k})> 1$. Moreover, if $R$ is Gorenstein, then $\curv_R(\Der_{k}(R))> 1$.
    \end{enumerate}
\end{theorem}

\begin{proof}
Note that $\mu(\Omega_{R/k})=\embdim(R)$, see, e.g., \cite[Cor.~6.5.(b)]{kunz}.
Moreover, $e(\Omega_{R/k})=e(R)\dim(R)$ by \Cref{rank}. Since $R$ is reduced local, if $R$ is Artinian, then $R$ is a field, and there is nothing to prove. So we may assume that $\dim(R)>0$. In view of \Cref{rank}, $\Omega_{R/k}$ has rank same as $\dim(R)>0$. In particular, it follows that $\dim(\Omega_{R/k})=\dim(R)$. Now $\Omega_{R/k}$ being CM implies that $\Omega_{R/k}$ is MCM. Note that $\Der_k(R)\cong \Hom_R(\Omega_{R/k}, R)$. So, if $R$ is Gorenstein, by \Cref{cx-duality}.(2), $\curv(\Omega_{R/k})=\injcurv(\Der_k(R))$ and $\injcurv(\Omega_{R/k})=\curv(\Der_k(R))$.

(1) Under each of the conditions, by the discussions made above, $e(\Omega_{R/k}) = e(R) \dim(R) \le 2 \embdim(R) = 2 \mu(\Omega_{R/k})$ and $\curv_R(\Omega_{R/k})\le 1$. Therefore, by \Cref{cor:char-CI-rings}.(1), one obtains that $R$ is complete intersection. Thus, in view of the discussion preceding \cite[Thm.~4.1]{herz}, $\Omega_{R/k}$ has finite projective dimension. Therefore, since $\Omega_{R/k}$ is MCM, it must be free. Hence $R$ is regular.

(2) The given hypotheses ensure that
$e(\Omega_{R/k})=e(R)\dim(R)\ge 2\embdim(R) = 2\mu(\Omega_{R/k})$. If possible, let $\injcurv_R(\Omega_{R/k})\le 1$ (equivalently, $\curv_R(\Der_{k}(R))\le 1$ under the condition that $R$ is Gorenstein). Then, by \Cref{cor:char-CI-minmult-rings}.(1), $R$ is complete intersection. Thus, following the proof of (1), one obtains that $R$ is regular, i.e., $e(R)=1$. Hence, since $R$ is not a field, $e(R)\dim(R) = \dim(R) < 2\dim(R) = 2\embdim(R)$, which is a contradiction. Therefore $\injcurv_R(\Omega_{R/k})> 1$ (equivalently, $\curv_R(\Der_{k}(R)) > 1$ when $R$ is Gorenstein).
\end{proof}


\Cref{thm:diff} in particular addresses the following celebrated conjecture due to Berger.

\begin{conjecture}[Berger]\cite{berg}\label{Conj-B63}
     Suppose $R$ is one-dimensional reduced local ring, the residue field $k$ of $R$ has characteristic $0$, and $R$ is either essentially of finite type over $k$ or it is a homomorphic image of a power series ring $k[[X_1,\dots,X_n]]$ over $k$. If $\Omega_{R/k}$ is torsion-free, then $R$ is regular.
\end{conjecture}
\begin{corollary}\label{cor:Berger}
      {\rm \Cref{Conj-B63}} holds true under each of the following conditions
      \begin{enumerate}[\rm(1)]
          \item 
          $e(R) \le 2 \embdim(R)$ and $\curv_R(\Omega_{R/k})\le 1$.
          \item 
          $R$ is Gorenstein, $e(R) \le 2 \embdim(R)$ and $\injcurv_R(\Der_k(R))\le 1$.
      \end{enumerate}    
\end{corollary}
\begin{proof}
 
Since $R$ is reduced, and $\dim(R)=1$, it follows that $R$ is CM. Therefore, $\Omega_{R/k}$ being torsion-free implies that $\Omega_{R/k}$ is MCM. Thus, we are done by \Cref{thm:diff}.(1).
\end{proof}

\begin{remark}
It is to be noted that when $R$ is a one-dimensional domain which is a homomorphic image of a power series ring $k[[X_1,\dots,X_n]]$ over a field $k$ of characteristic $0$, then $e(R)\le 2\embdim(R)$ and $\Omega_{R/k}$ being torsion-free imply that $\embdim(R)\le 5$, see \cite[Thm.~1]{Is}. 
\end{remark}

\section*{Acknowledgments}
We thank Sarasij Maitra for helpful discussions regarding literature on module of differentials.
Souvik Dey was partly supported by the Charles University Research Center program No.UNCE/24/SCI/022 and a grant GA \v{C}R 23-05148S from the Czech Science Foundation.
Aniruddha Saha was supported by Senior Research Fellowship (SRF) from UGC, MHRD, Govt.\,of India.

\end{document}